\documentclass{amsart}

\usepackage[margin=1.5in]{geometry}
\usepackage{amsmath}
\usepackage{amssymb}
\usepackage{amsthm}
\usepackage{enumerate}
\usepackage[colorlinks=true,
citecolor=blue,
linkcolor=red,
urlcolor=black,
anchorcolor=black]{hyperref}
\usepackage{xcolor}
\usepackage[english]{babel}
\usepackage{graphicx}
\usepackage{longtable}
\usepackage{tikz-cd}

\numberwithin{equation}{section}
\newtheorem{theorem}{Theorem}[section]
\newtheorem{prop}[theorem]{Proposition}
\newtheorem{lemma}[theorem]{Lemma}

\newtheorem*{claim}{Claim}

\theoremstyle{definition}

\newcommand{\Z}{\mathbb{Z}}

\newcommand{\Q}{\mathbb{Q}}

\newcommand{\R}{\mathbb{R}}

\newcommand{\calQ}{\mathcal{Q}}

\newcommand{\Qsquare}{\Q^\times / (\Q^\times)^2}
\newcommand{\Qpxsquare}{(\Q_p^\times)^2}

\title[Low-dimensional cusp cross-sections]{Cusp cross-sections of low-dimensional arithmetic hyperbolic manifolds}
\author{Marcus Grimbert}
\address{École Normale supérieure de Lyon}
\email{marcus.grimbert@ens-lyon.fr}
\author{Duncan McCoy}%
\address {Université du Québec à Montréal}
\email{mc\_coy.duncan@uqam.ca}
\author{Connor Sell}%
\address{Université du Québec à Montréal}
\email{sell.connor@gmail.com }

\date{}

\begin{document}

\begin{abstract}
We study which compact flat manifolds of dimensions $3 \leq n \leq 6$ occur as cusp cross-sections in commensurability classes of non-compact arithmetic hyperbolic $(n+1)$-manifolds. Using the classification of low-dimensional flat manifolds together with the relationship between their holonomy representations and rational quadratic forms, we determine the possible arithmetic commensurability classes for every flat manifold in these dimensions. We show that every flat $n$-manifold for $n=3,4,5$ occurs as a cusp cross-section in infinitely many distinct arithmetic commensurability classes. In dimension six, the same holds with exactly eight exceptions: eight non-orientable flat $6$-manifolds occur in a unique arithmetic commensurability class. We also show that, for every $3 \leq n \leq 6$, there is a single commensurability class of arithmetic hyperbolic $(n+1)$-manifolds containing every flat $n$-manifold as a cusp cross-section.
\end{abstract}
\maketitle

\section{Introduction}
The aim of this paper is to explore phenomena around the cusp cross-sections of arithmetic hyperbolic manifolds, using the classification of flat manifolds in low dimensions \cite{Cid2001Computation}. By the work of Long-Reid and McReynolds it is known that every compact flat manifold arises as a cusp cross-section in at least one commensurability class of \emph{arithmetic} hyperbolic manifolds \cite{LongReid, McReynolds2009covers}. In the other direction, there is an algebraic condition which allows one to characterize precisely which commensurability classes of arithmetic hyperbolic manifolds can contain a given flat manifold as a cusp cross-section \cite{McCoySell2024}. Using this criterion, one can show that in all sufficiently large dimensions there are flat manifolds that arise as a cusp cross-section in a unique commensurability class of arithmetic hyperbolic manifolds \cite{McCoySell2025}. We say that a flat $n$-manifold has the \emph{unique arithmetic commensurability class (UCC) property} if it occurs as a cusp cross-section in a unique commensurability class of non-compact arithmetic hyperbolic $(n+1)$-manifolds. For example, since there is a unique commensurability class of non-compact arithmetic hyperbolic surfaces, $S^1$ has the UCC property.

\begin{theorem}\label{thm:UCC_analysis}\quad
\begin{itemize}
    \item For $n=3,4,5$, all flat $n$-manifolds occur as cusp cross-sections in infinitely many commensurability classes of arithmetic hyperbolic $(n+1)$-manifolds.
    \item With eight exceptions, all flat 6-manifolds occur as cusp cross-sections in infinitely many commensurability classes of arithmetic hyperbolic 7-manifolds. The eight exceptions are all non-orientable and have the UCC property.
\end{itemize}
\end{theorem}
The existence of a non-orientable flat 6-manifold with the UCC property was established in \cite{McCoySell2025}. Theorem~\ref{thm:UCC_analysis} proves, however, that there does not exist an orientable flat 6-manifold with the UCC property.

Another known phenomenon is that in all sufficiently large dimensions there are pairs of flat manifolds $B_1$ and $B_2$ that can never appear as cusp cross-sections in the same commensurability class of arithmetic hyperbolic manifolds \cite{McCoySell2025}. It turns out that this phenomenon does not arise in low dimensions. In fact, in all dimensions $n\leq 6$ there is a commensurability class of arithmetic hyperbolic $(n+1)$-manifolds containing every flat $n$-manifold as a cusp cross-section. 
\begin{theorem}\label{thm:universal_class}
    Let $B$ be a compact flat manifold of dimension $3\leq n\leq 6$. Then $B$ occurs as a cusp cross-section in the commensurability class of arithmetic hyperbolic manifolds defined by the quadratic form
    \[
    q_0(x)=x_1^2 +\dots + x_{n+1}^2-x_{n+2}^2.
    \]
\end{theorem}

Finally, we gather some statistics on the proportion of flat manifolds that appear as cusp cross-sections in every commensurability class of arithmetic hyperbolic manifolds. Previous results had shown that of the six orientable flat 3-manifolds, three of them (the 3-torus, the half-twist manifold and the Hantzsche-Wendt manifold) occur as cusp cross-sections in every commensurability class of arithmetic hyperbolic 4-manifolds \cite{Sell}. Similarly it was known that 12 of 27 compact orientable flat 4-manifolds occur as cusp cross-sections in every commensurability class of arithmetic hyperbolic 5-manifolds \cite{McCoySell2024}. The data from the flat 5- and 6-manifolds motivates the question of whether the proportion of flat $n$-manifolds (both orientable and non-orientable) that appear as cusp cross-sections in every commensurability class of arithmetic hyperbolic $(n+1)$-manifolds tends to one.
\begin{table}[]
    \centering
    \begin{tabular}{|c|c|c|} \hline
		Dimension & Orientable  & All \\ \hline
            $3$ & $3/6=50\%$ & $7/10=70\%$  \\ \hline
            $4$ & $12/27=44.4\%$ & $44/74=59.5\%$ \\ \hline
            $ 5$ & $153/174=87.9\%$ & $1023/1060=96.5\%$  \\ \hline
            $ 6 $ & $2868/3314=86.5\%$ & $37216/38746=96.1\%$ \\ \hline
	\end{tabular}
    \caption{Statistics on the flat manifolds appearing as cusp cross-sections in all commensurability classes of non-compact arithmetic hyperbolic manifolds.}
    \label{tab:UCC_constructions}
\end{table}

\subsection*{Acknowledgements}
DM is partially supported by NSERC grant RGPIN-2020-05491 and a Canada Research Chair.

\section{Preliminaries}

\subsection{Rational quadratic forms}
Let $f: \Q^n \rightarrow \Q$ be a non-degenerate rational quadratic form. 
The orthogonal group of $f$ is the matrix group
\[ 
O(f;\Q)=\{A\in GL_n(\Q) \mid f(Av)=f(v)\, \forall v\in \Q^n\}.
\]
Two rational quadratic forms $f$ and $g$ are (rationally) equivalent if there exists a matrix $C\in GL_n(\Q)$ such that
\[
\text{$f(v)=g(Cv)$ for all $v\in \Q^n$.}
\]
We recall the rational equivalence class of a non-degenerate rational quadratic form $f$ is completely determined by the combination of its signature, discriminant $d(f)\in \Qsquare$ and Hasse-Witt invariants $\varepsilon_p(f)\in \{\pm1\}$ at each prime $p$ \cite{Serre_arithmetic}.
Two rational quadratic forms $ f $ and $ g $ are \emph{projectively equivalent} if there exists a positive $ m \in \mathbb{Q}_{>0} $ such that $ mf $ and $ g $ are rationally equivalent. The projective equivalence class of a rational quadratic form can also be characterized in terms of the discriminant and Hasse-Witt invariants with the precise characterization depending on the rank modulo four.
\begin{prop}[{\cite[Proposition~5.4]{McCoySell2024}}]
\label{prop:projective_invariants}
    Let $f$ and $g$ be two non-degenerate quadratic forms of rank $n$ over $\Q$. Then $f$ and $g$ are projectively equivalent if and only if they have the same signature and one of the following conditions holds:
    \begin{enumerate}[(i)]
        \item $n\equiv 0 \bmod 4$, $d(f)=d(g)=d$ and $\varepsilon_p(f)=\varepsilon_p(g)$ for all primes $p$ such that $d\in \Qpxsquare$;
        \item $n\equiv 1 \bmod 4$ and $\varepsilon_p(f)=\varepsilon_p(g)$ for all primes $p$;
        \item $n\equiv 2 \bmod 4$, $d(f)=d(g)=d$ and $\varepsilon_p(f)=\varepsilon_p(g)$ for all primes $p$ such that $-d\in \Qpxsquare$ or
        \item $n\equiv 3 \bmod 4$ and $(d(f),-1)_p\varepsilon_p(f)=(d(g),-1)_p\varepsilon_p(g)$ for all primes $p$.
    \end{enumerate}
   \qed
\end{prop}

\subsection{Holonomy forms}
Let $B$ be a compact flat $n$-manifold with fundamental group $\Gamma$. Then $\Gamma$ is a torsion-free group which we may identify as a discrete subgroup of $O(n)\ltimes \R^n$. The translation subgroup $T\leq \Gamma$ of elements acting by translations on $\R^n$ is a normal abelian subgroup. Moreover, the first Bieberbach theorem implies that $T$ is a maximal abelian subgroup of $\Gamma$ of finite index and that it is isomorphic to $\Z^n$. Thus identifying $T\cong \Z^n$ yields a short exact sequence
\begin{equation}
    1\rightarrow \Z^n \rightarrow \Gamma \rightarrow H \rightarrow 1,
\end{equation}
where the finite group $H$ is the \emph{holonomy group} of $B$. Letting $H$ act on $\Z^n$ by conjugation yields a representation
\[
\rho_{\mathrm{hol}} : H \rightarrow GL_n(\Z)
\]
which is well-defined up to conjugation by elements of $GL_n(\Z)$.
Since the translation subgroup is a maximal abelian subgroup of $\Gamma$, this representation is faithful. We define a (rational) \emph{holonomy representation} for $B$ to be any rational representation
\[
\rho : H \rightarrow GL_n(\Q)
\]
in the conjugacy class of representations defined by $\rho_{\mathrm{hol}}$.

A \emph{holonomy form} for $B$ is a positive definite rational quadratic form $f$ of rank $n$ such that $B$ has a holonomy representation 
\[
\rho:H \rightarrow O(f;\Q).
\]
It is not hard to check that the set of all possible holonomy forms for a flat manifold $B$ is closed under projective equivalence.

In order to calculate the holonomy forms of a given flat manifold it will often be convenient to decompose the holonomy representation into subrepresentations.
\begin{prop}[{\cite[Proposition 6.1]{McCoySell2024}}]\label{prop:form_decomposition}
    Let $\rho=\sigma_1 \oplus \dots \oplus \sigma_\ell$ be a direct sum of rational representations of a finite group $G$. If the image of $\rho$ is contained in $O(f;\Q)$, for a non-degenerate rational quadratic form $f$, then $f$ is equivalent to
    \[
    f\cong g_1 \oplus \dots \oplus g_\ell,
    \]
    where each $g_i$ is a quadratic form such that the image of $\sigma_i$ is contained in $O(g_i;\Q)$.
    \qed
\end{prop}

\subsection{Arithmetic hyperbolic manifolds}
For any rational quadratic form $q$ of signature $(n+1,1)$, there is an associated commensurability class of arithmetic hyperbolic $(n+1)$-manifolds obtained by considering subgroups of $O(q;\R)$ commensurable to $O(q;\Z)$. Moreover two rational quadratic forms $q_1$ and $q_2$ of signature $(n+1,1)$ define the same commensurability class of arithmetic hyperbolic manifolds if and only if $q_1$ and $q_2$ are projectively equivalent \cite[2.6]{Gromov}.
The main result of \cite{McCoySell2024} determines precisely when a given flat manifold $B$ arises as a cusp cross-section in a commensurability class of arithmetic hyperbolic manifolds.
\begin{theorem}[{\cite[Theorem~1.1]{McCoySell2024}}]\label{thm:realization}
    Let $B$ be a compact flat $n$-manifold and let $q$ be a rational quadratic form of signature $(n+1,1)$. Then $B$ arises as a cusp cross-section in the commensurability class of arithmetic hyperbolic manifolds defined by $q$ if and only if $q$ is projectively equivalent to $f\oplus \langle 1,-1\rangle$ where $f$ is a holonomy form for $B$.\qed
\end{theorem}
It turns out that two rational quadratic forms $f_1$ and $f_2$ are projectively equivalent if and only if the forms $f_1\oplus \langle 1,-1\rangle$ and $f_2\oplus \langle 1,-1\rangle$ are projectively equivalent.\footnote{This follows from Witt cancellation and the observation that $\langle m, -m\rangle$ and $\langle 1,-1 \rangle$ are rationally equivalent for all $m\neq 0$:
\[
X^2 -Y^2=mx^2-my^2
\]
for
$X=\frac{1}{2}((1+m)x+(1-m)y)$ and $Y=\frac{1}{2}((1-m)x+(1+m)y)$}
Consequently, understanding the commensurability classes of arithmetic hyperbolic manifolds containing a given flat manifold $B$ as a cusp cross-section is equivalent to understanding its holonomy forms up to projective equivalence.

\section{Crystal families and their invariant forms}\label{sec:crystal_forms}
In order to understand the holonomy forms of a flat manifold it will be useful to think in terms of crystal families of representations. For $i=1,2$ let
\[
\rho_i: G_i \rightarrow GL_n(\Q)
\]
be two rational representations of finite groups. Then $\rho_1$ and $\rho_2$ belong to the same \emph{crystal family} if we may conjugate $\rho_1$ over $\Q$ so that the sets
\[
\calQ_i=\{Q\in M_{n\times n}(\Q)|\text{$Q^T=Q$ and $\rho_i(g)^TQ\rho_i(g)=Q$ for all $g\in G_i$}\}
\]
are the same for $i=1,2$. Equivalently $\rho_1$ and $\rho_2$ lie in the same crystal family if and only if they preserve exactly the same rational quadratic forms up to conjugation.

There are 12 crystal families of $\Q$-irreducible representations of dimension $n\leq 5$ \cite{Plesken84}. We use the naming convention for these families of irreducible representations as found in \cite{Plesken84}.
\begin{center}
\begin{tabular}{|c|c|}
\hline
    Dimension & Crystal families  \\
    \hline
      1& 1\\ 
      2& 2-1, 2-2\\
      3& 3\\ 
      4& 4-1, 4-1', 4-2, 4-2', 4-3, 4-3'\\
      5& 5-1, 5-2\\
      \hline
\end{tabular}
\end{center}
 In practice, only 10 of these crystal families arise as subrepresentations in the holonomy representations of flat manifolds of dimensions $n\leq 6$, with 4-3 and 5-2 being the pair of crystal families that do not appear \cite{Cid2001Computation}.

For each of these irreducible crystal families, we will need to know the set of positive definite rational quadratic forms $f$ such that $O(f;\Q)$ contains the image of a representation in that crystal family. For the crystal families containing $\R$-irreducible representations this is a straight-forward task. For any $\R$-irreducible representation $\rho$ the vector space of real quadratic forms preserved by $\rho$ is one-dimensional (see \cite[Lemma~VI.1]{Plesken81}, for example). Thus for crystal families that contain $\R$-irreducible representations, there is a single projective equivalence class of positive definite quadratic forms $f$ such that $O(f;\Q)$ contains representations from that family. This projective equivalence class can be found by taking one representation from the crystal family and finding a single positive definite rational quadratic form $f$ such that its image is contained in $O(f;\Q)$.

Fortunately, the majority of the crystal families we consider contain $\R$-irreducible representations. In fact we will need to consider only three families, 4-1', 4-2' and 4-3',  that do not contain $\R$-irreducible representations. For these families the space of invariant quadratic forms is 2-dimensional and the calculation is more involved.

We summarize necessary results for these families in the subsequent sections with proofs and calculations deferred to the appendices (Appendix~\ref{sec:R_irred} for the $\R$-irreducible case and Appendix~\ref{sec:R_red} for the three families that are reducible over $\R$). We include information on the families 4-3 and 5-2 for completeness.

\subsection{The 1-dimensional family}
There is a unique $1$-dimensional crystal family denoted simply by $1$. Clearly, given a representation $\rho: G\rightarrow GL_1(\Q)=\Q^\times$ with family symbol $1$, its image is contained in $O(q;\Q)$ for every 1-dimensional quadratic form.
\subsection{The 2-dimensional families}
There are two distinct irreducible $2$-dimensional crystal families. The first, denoted 2-1, contains representations whose image contains an element of order four. The second, denoted 2-2, contains the representations whose image contains an element of order three. Both families contain $\R$-irreducible representations. Analysis such as that found in \cite[Lemma~7.6]{McCoySell2024} yields the following result.
 \begin{lemma}\label{lem:dim2analysis}
     Let $\rho:G\rightarrow GL_2(\Q)$ be an irreducible representation. Let $f$ be a positive-definite rational quadratic form. Then $\rho$ is conjugate to a representation with image in $O(f;\Q)$ if and only if
     \begin{enumerate}[(a)]
         \item\label{it:2-1} $\rho$ has family symbol 2-1 and $f$ satisfies $d(f)=1$ and $\varepsilon_p(f)=1$ for all $p\equiv 1 \bmod 4$ or
         \item\label{it:2-2} $\rho$ has family symbol 2-2 and $f$ satisfies $d(f)=3$ and $\varepsilon_p(f)=1$ for all $p\equiv 1 \bmod 3$.
     \end{enumerate}\qed
 \end{lemma}
\subsection{The 3-dimensional family}
There is a unique $3$-dimensional crystal family denoted by $3$. This family contains every $\Q$-irreducible 3-dimensional rational representation. If we consider, for example, the faithful 3-dimensional representation of $A_4$, we see that the representations in this family are $\R$-irreducible. Furthermore, since $A_4$ admits an orthogonal rational representation, we obtain the following result.
\begin{lemma}\label{lem:dim3analysis}
     Let $\rho:G\rightarrow GL_3(\Q)$ be an irreducible representation. Let $f$ be a positive-definite rational quadratic form. Then $\rho$ is conjugate to a representation with image in $O(f;\Q)$ if and only if $(d(f),-1)_p\varepsilon_p(f)=1$ for all $p$. \qed
 \end{lemma}
 \subsection{The 4-dimensional families}
 There are six crystal families of irreducible $4$-dimensional representations, denoted 4-1, 4-2, 4-3, 4-1', 4-2' and 4-3'. The families 4-1, 4-2 and 4-3 are all $\R$-irreducible, whereas 4-1', 4-2' and 4-3' all decompose over $\R$ into two distinct $\R$-irreducible subrepresentations. Representatives for each family can be described as follows.
 \begin{itemize}
     \item The family 4-1 contains the faithful rational representation of the quaternion group $Q_8$.
     \item The family 4-2 contains the faithful rational representation of $C_3\rtimes C_4$.
     \item The family 4-3 contains the irreducible 4-dimensional representations of $A_5$.
     \item The family 4-1' contains the faithful rational representation of $C_8$.
     \item The family 4-2' contains the irreducible faithful rational representation of $C_{12}$.
     \item The family 4-3' contains the faithful rational representation of $C_{5}$.
 \end{itemize}
 The following lemma summarizes the quadratic forms invariant under each of these crystal families. The analysis of the forms 4-1', 4-2' and 4-3' is reserved to the appendix.
 \begin{lemma}\label{lem:dim4analysis}
     Let $\rho:G\rightarrow GL_4(\Q)$ be a $\Q$-irreducible representation. Let $f$ be a positive-definite rational quadratic form. Then $\rho$ is conjugate to a representation with image in $O(f;\Q)$ if and only if
     \begin{enumerate}[(a)]
         \item\label{it:4-1} $\rho$ has family symbol 4-1 and $f$ satisfies $d(f)=1$ and $\varepsilon_p(f)=1$ for all $p$;
         \item\label{it:4-2} $\rho$ has family symbol 4-2 and $f$ satisfies $d(f)=1$ and $\varepsilon_p(f)=(3,3)_p$ for all $p$;
         \item\label{it:4-3} $\rho$ has family symbol 4-3 and $f$ satisfies $d(f)=5$ and $\varepsilon_p(f)=1$ for all $p\equiv 1,4\bmod 5$;
         \item\label{it:4-1'} $\rho$ has family symbol 4-1' and $f$ satisfies $d(f)=1$ and $\varepsilon_p(f)=1$ for all $p\equiv 1,3,5 \bmod 8$;
         \item\label{it:4-2'} $\rho$ has family symbol 4-2' and $f$ satisfies $d(f)=1$ and $\varepsilon_p(f)=1$ for all $p\equiv 1,5,7 \bmod 12$;
         \item\label{it:4-3'} $\rho$ has family symbol 4-3' and $f$ satisfies $d(f)=5$ and $\varepsilon_p(f)=1$ for all $p\equiv 1\bmod 5$.
     \end{enumerate}
 \end{lemma}

\subsection{The 5-dimensional families}
There are two crystal families of irreducible $5$-dimensional representations, denoted 5-1 and 5-2. Both of these contain $\R$-irreducible representations. Representatives for each family can be described as follows.
 \begin{itemize}
     \item The family 5-1 contains the natural faithful representation of the semi-direct product $C_2^5 \rtimes C_5$.
     \item The family 5-2 contains the 5-dimensional representation of $A_6$.
 \end{itemize}
Both of these families consist of $\R$-irreducible representations.

  \begin{lemma}\label{lem:dim5analysis}
     Let $\rho:G\rightarrow GL_5(\Q)$ be a $\Q$-irreducible representation. Let $f$ be a positive-definite rational quadratic form. Then $\rho$ is conjugate to a representation with image in $O(f;\Q)$ if and only if
     \begin{enumerate}[(a)]
         \item\label{it:5-1} $\rho$ has family symbol 5-1 and $f$ satisfies $\varepsilon_p(f)=1$ for all $p$;
         \item\label{it:5-2} $\rho$ has family symbol 5-2 and $f$ satisfies $\varepsilon_p(f)=(3,3)_p$ for all $p$;
     \end{enumerate}\qed
 \end{lemma}

\section{Results}
Given a flat manifold $B$, we describe our procedure for determining which commensurability classes of arithmetic hyperbolic manifolds contain $B$ as a cusp cross-section. Given a flat $n$-manifold $B$ in the database of low-dimensional flat manifolds, there is a presentation of its fundamental group as a subgroup of $GL_n(\Z)\ltimes \Q^n$, where the translation subgroup is $\Z^n$.\footnote{Although these manifolds can be generated using the CARAT package \cite{Carat2.1b1}, we used the tables available online at \cite{LutowskiCARAT}.} Using this description one may write down a holonomy representation $\rho_{\rm hol}$ for $B$. We then determine the unordered list $\mathcal{C}=(C_1, \dots, C_k)$ of crystal families such that $\rho_{\rm hol}$ can be decomposed as
\[
\rho_{\rm hol} \cong \rho_1 \oplus \dots \oplus \rho_k,
\]
where each $\rho_i$ is a rational subrepresentation that is irreducible over $\Q$ and $C_i$ is the crystal family containing $\rho_i$.

By Proposition~\ref{prop:form_decomposition} and Theorem~\ref{thm:realization}, $B$ can appear as a cusp cross-section in the commensurability class of arithmetic hyperbolic manifolds defined by a form $q$ if and only if $q$ is equivalent to a form
\[
q\cong f_1 \oplus \dots \oplus f_k \oplus \langle 1,-1\rangle,
\]
where $f_i$ is positive definite and $\rho_i$ has image in $O(f_i;\Q)$.

In particular, for a given $\rho_i$ the set of possible forms $f_i$ is determined by the crystal family of $\rho_i$. Consequently, the set of commensurability classes of arithmetic hyperbolic manifolds containing $B$ as a cusp cross-section is determined by the list of crystal families $\mathcal{C}$. Moreover, using the calculations of Section~\ref{sec:crystal_forms}, the list of crystal families $\mathcal{C}$ can be converted into a classification of all the commensurability classes of arithmetic hyperbolic manifolds containing $B$ as a cusp cross-section.

In practice, there are a relatively small set of possibilities for the list $\mathcal{C}$ that can arise for low-dimensional flat manifolds. As the cusp-occurrence of a flat manifold in commensurability classes of arithmetic manifolds is completely governed by its corresponding list $\mathcal{C}$, this allows us to partition the low-dimensional flat manifolds into a relatively small set of different types according to their cusp-occurrence in arithmetic manifolds. 
\begin{center}
\begin{tabular}{|c|c|c|}
\hline
    Dimension & \# possibilities for $\mathcal{C}$ & \# cusp-occurrence types \\
    \hline
    3 & 3 & 3 \\
    4 & 4 & 4\\
    5 & 10 & 7 \\
    6 & 17 & 10 \\
    \hline
\end{tabular}
\end{center}

If $\mathcal{C}$ contains three or more odd-dimensional crystal families, then $B$ appears as a cusp cross-section in every commensurability class of arithmetic hyperbolic manifolds \cite[Theorem~7.1]{McCoySell2024}.
For the remaining lists containing at most two odd-dimensional families (of which there are 24 possibilities), the derivation of the corresponding cusp-occurrence behaviour was carried out by hand. We omit the full details of these calculations as they are routine and somewhat repetitive. For the 3- and 4-dimensional flat manifolds these calculations are implicit in \cite[Section~7.3]{McCoySell2024}. For 5- and 6-dimensional flat manifolds we illustrate the flow of these calculations via several examples. 

\subsection{Deriving cusp-occurrence from the crystal family}
We now give some sample calculations illustrating how the cusp-occurrence conditions of Theorems~\ref{thm:dim3}--\ref{thm:dim6} arise. 
\subsubsection*{Example $\mathcal C=(2\text{-}1,2\text{-}2,1)$}
Suppose that $B$ is a flat $5$-manifold whose holonomy representation
has crystal-family decomposition
\[
    \mathcal C=(2\text{-}1,2\text{-}2,1).
\]
We explain the occurrence of the congruence condition in
type~\eqref{it:dim5 2-1,2-2} of Theorem~\ref{thm:dim5}.

By Proposition~\ref{prop:form_decomposition}, a holonomy form for $B$ decomposes as
\[
    f=f_1\oplus f_2\oplus\langle a\rangle,
\]
where $f_1$ is invariant under a representation of family
$2\text{-}1$, $f_2$ is invariant under a representation of family
$2\text{-}2$, and $a\in\mathbb Q_{>0}$.

By Lemma~\ref{lem:dim2analysis},
\[
    d(f_1)=1,
    \qquad
    \varepsilon_p(f_1)=1
    \quad\text{for }p\equiv1\bmod4,
\]
whereas
\[
    d(f_2)=3,
    \qquad
    \varepsilon_p(f_2)=1
    \quad\text{for }p\equiv1\bmod3.
\]
Now suppose that $p\equiv1\bmod{12}$. Then both $-1$ and $3$ are
squares in $\mathbb Q_p$. Consequently all the cross terms occurring
in the Hasse--Witt invariant of
$f_1\oplus f_2\oplus\langle a\rangle$ are trivial at $p$, and hence $\varepsilon_p(f)=1$.

Let
\[
    q=f\oplus\langle1,-1\rangle.
\]
Since
\[
    d(q)=-d(f)
\]
and
\[
    \varepsilon_p(q)
       =\varepsilon_p(f)(d(f),-1)_p,
\]
we obtain, for $p\equiv1\bmod{12}$,
\[
\begin{aligned}
    (d(q),-1)_p\varepsilon_p(q)
      &=
      (-d(f),-1)_p(d(f),-1)_p  \\
      &=
      (-1,-1)_p
       =1.
\end{aligned}
\]
Thus the holonomy decomposition
$\mathcal C=(2\text{-}1,2\text{-}2,1)$ imposes the condition
\[
    (d(q),-1)_p\varepsilon_p(q)=1
    \qquad\text{for every }p\equiv1\bmod{12},
\]
which is the condition appearing in type~\eqref{it:dim5 2-1,2-2} of Theorem~\ref{thm:dim5}.

Conversely, we need to verify that every form $q$ of signature $(6,1)$ satisfying type~\eqref{it:dim5 2-1,2-2} of Theorem~\ref{thm:dim5} is projectively equivalent to the form
\[
q\cong f_1 \oplus f_2 \oplus \langle a \rangle \oplus \langle 1,-1 \rangle
\]
where $a\in \Q_{>0}$ and $f_1$ and $f_2$ are invariant under representations of families $2\text{-}1$ and  $2\text{-}2$ respectively.

Suppose that we have such a form $q$. Let $d(q)=-d$ for $d\in \Q_{>0}$. Since $(d(q),-1)_p=1$ whenever $p\equiv 1 \bmod 12$, we have that $\varepsilon_p(q)=1$ for $p\equiv 1\bmod 12$.
By Lemma~\ref{lem:dim2analysis}, there exist positive definite forms $f_1$, $f_2$ invariant under representations of families $2\text{-}1$ and  $2\text{-}2$ respectively with the following invariants:
\begin{itemize}
    \item $f_1$ satisfies $d(f_1)=1$ and 
    \[\varepsilon_p(f_1)=\begin{cases}
        1& p\equiv 1,3,5 \bmod 12\\
        \varepsilon_p(q)(-3d(q),-1)_p & p\equiv 7,11 \bmod 12 
    \end{cases}
    \]
    \item $f_2$ satisfies $d(f_2)=3$ and 
    \[\varepsilon_p(f_2)=\begin{cases}
        1& p\equiv 1,7,11 \bmod 12\\
        \varepsilon_p(q)(-3d(q),-1)_p & p\equiv 3,5 \bmod 12 
    \end{cases}
    \]
\end{itemize}
Note that the value of $\varepsilon_2(f_i)$ will be determined by the values of $\varepsilon_p(f_i)$ for $p$ odd which have all been determined by the above conditions. 
    
We take $q'$ to be the form 
    \[
    q'\cong f_1 \oplus f_2 \oplus \langle 3 \rangle \oplus \langle 1,-1 \rangle.
    \]
    For $q'$ we have that $d(q')=-1\in \Qsquare$ and for any odd prime $p$
    \begin{align*}
        (d(q'),-1)_p\varepsilon_p(q')&=(-1,-1)_p\varepsilon_p(f_1)\varepsilon_p(f_2)(3,-3)_p(3,-1)_p\\
        &=\begin{cases}
        (-3,-1)_p& p\equiv 1 \bmod 12\\
        \varepsilon_p(q)(-3,-1)_p(-3d(q),-1)_p & p\equiv 3,5,7,11 \bmod 12 
    \end{cases}\\
    &= (d(q),-1)_p\varepsilon_p(q),
    \end{align*}
using the fact that $(d(q),-1)_p=(-3,-1)_p=1=\varepsilon_p(q)$ for $p\equiv 1 \bmod 12$. Thus we see that $(d(q),-1)_p\varepsilon_p(q)=(d(q'),-1)_p\varepsilon_p(q')$ for all odd primes. By Proposition~\ref{prop:projective_invariants}, this implies that $q$ and $q'$ are projectively equivalent. Thus $q$ is projectively equivalent to a quadratic form in the desired form.
\subsubsection*{Example $\mathcal C=(4\text{-}2,1,1)$}
Consider a flat $6$-manifold $B$ whose holonomy representation has
crystal-family decomposition
\[
    \mathcal C=(4\text{-}2,1,1).
\]
We show directly that this gives type~\eqref{it:dim6 4-2} of Theorem~\ref{thm:dim6}.

By Proposition~\ref{prop:form_decomposition} and Theorem~\ref{thm:realization}, $B$ appears as a cusp cross-section in the commensurability class defined by a form $q$ if and only if $q$ is equivalent to 
\[
    q\cong g\oplus\langle a,b\rangle \oplus \langle 1,-1\rangle
\]
where $g$ satisfies the conditions of Lemma~\ref{lem:dim4analysis}\eqref{it:4-2} and $a,b\in\mathbb Q_{>0}$.

Since $d(g)=1$, we have $d(q)=-ab$ and calculating the Hasse--Witt invariants gives
\[
\begin{aligned}
    \varepsilon_p(q)
      &=\varepsilon_p(g)(a,b)_p(ab,-1)_p \\
      &=(3,3)_p(a,b)_p(ab,-1)_p.
\end{aligned}
\]
Now suppose that $d(q)\in(\mathbb Q_p^\times)^2$. Since $d(q)=-ab$, this implies
\[
    ab\equiv -1
    \bmod{(\mathbb Q_p^\times)^2}.
\]
Consequently
\[
    (a,b)_p=(a,-a)_p=1 \qquad\text{and}\qquad (ab,-1)_p=(-1,-1)_p.
\]
It follows that
\[
    \varepsilon_p(q)
      =(-1,-1)_p(3,3)_p
      =(-3,-1)_p
\]
whenever $d(q)$ is a square in $\mathbb Q_p$. That is, we have the defining condition for type~\eqref{it:dim6 4-2} of Theorem~\ref{thm:dim5}

Conversely, suppose that $q$ is a rational quadratic form of signature
$(7,1)$ satisfying
\[
    \varepsilon_p(q)=(-1,-1)_p(3,3)_p=(-3,-1)_p
\]
for every prime $p$ such that
$d(q)\in(\mathbb Q_p^\times)^2$. Write $-d=d(q)$ and consider a positive definite form $f$ such that $d(f)=1$ and $\varepsilon_p(f)=(3,3)_p$ for all $p$ (for example, the diagonal form $f=\langle 3,3,1,1\rangle$). Take
\[
    q'=f\oplus \langle 1,d\rangle\oplus\langle1,-1\rangle,
\]
then $d(q')=-d=d(q')$, and
\begin{align*}
    \varepsilon_p(q')&=\varepsilon_p(f)(d,-1)_p\\
    &=(3,3)_p(d,-1)_p
\end{align*}
In particular, whenever $-d=d(q)=d(q')$ is a square in $\Q_p$ we have that 
\[\varepsilon_p(q')=(3,3)_p(-1,-1)_p=\varepsilon_p(q).\]
Proposition~\ref{prop:projective_invariants} therefore implies that $q$ and $q'$ are projectively
equivalent.

Thus a flat $6$-manifold with
\[
    \mathcal C=(4\text{-}2,1,1)
\]
occurs as a cusp cross-section in the commensurability class defined by
$q$ if and only if
\[
    \varepsilon_p(q)=(-1,-1)_p(3,3)_p
\]
for every prime $p$ such that
$d(q)\in(\mathbb Q_p^\times)^2$. This is the condition defining
type~\eqref{it:dim6 4-2}.

\subsubsection*{Examples $\mathcal C=(4\text{-}1,2\text{-}1)$ and $\mathcal C=(2\text{-}1,2\text{-}1,2\text{-}1)$}
Consider a flat $6$-manifold $B$ whose holonomy representation has
crystal-family decomposition
\[
    \mathcal C=(4\text{-}1,2\text{-}1).
\]
We show directly that this gives type~\eqref{it:dim6 2-1,2-1,2-1} of Theorem~\ref{thm:dim6}.

By Theorem~\ref{thm:realization} and Proposition~\ref{prop:form_decomposition} $B$ can appear as a cusp cross-section in the commensurability class of arithmetic hyperbolic manifolds defined by $q$ if and only if $q$ is projectively equivalent to
\[
q\cong f\oplus g \oplus \langle 1,-1\rangle
\]
for $f$ and $g$ that are positive definite forms that are invariant under representations of families $4\text{-}1$ and $2\text{-}1$ respectively.

By Lemma~\ref{lem:dim4analysis}\eqref{it:4-1} the form $f$ satisfies $d(f)=1$ and $\varepsilon_p(f)=1$ for all $p$. By Lemma~\ref{lem:dim2analysis}\eqref{it:2-1}, the form $g$ satisfies $d(g)=1$ and $\varepsilon_p(g)=1$ for all $p\equiv 1 \bmod 4$. For form $q$ as above we have that $d(q)=-1$ and that
\[
\varepsilon_p(q)=\varepsilon_p(f)\varepsilon_p(g)=\varepsilon_p(g)
\]
Consequently, $\varepsilon_p(q)=1$ for all $p\equiv 1 \bmod 4$, as required by condition~\eqref{it:dim6 2-1,2-1,2-1} of Theorem~\ref{thm:dim6}. Conversely, Proposition~\ref{prop:projective_invariants} and the fact that $p\equiv 1\bmod 4$ if and only if $-1\in\Qpxsquare$ shows that all such $q$ lie in a unique projective equivalence class.

It follows that every flat $6$-manifold with
$\mathcal C=(4\text{-}1,2\text{-}1)$ occurs as a cusp cross-section in
a unique arithmetic commensurability class. This is precisely
type~\eqref{it:dim6 2-1,2-1,2-1} of Theorem~\ref{thm:dim6}.

A similar calculation applies when
\[
    \mathcal C=(2\text{-}1,2\text{-}1,2\text{-}1).
\]

In this case, we consider forms $q$ that are projectively equivalent to
\[
    q\cong f_1 \oplus f_2 \oplus f_3 \oplus \langle 1,-1 \rangle,
\]
where each $f_i$ is positive definite and invariant under a representation in family 2-1. By Lemma~\ref{lem:dim2analysis}\eqref{it:2-1}, $f_1$, $f_2$ and $f_3$ satisfy $d(f_i)=1$ and $\varepsilon_p(f_i)=1$ for all $p\equiv 1 \bmod 4$. It is then easy to calculate that $d(q)=-1$ and that $\varepsilon_p(q)=1$ for all $p\equiv 1 \bmod 4$.  Thus this crystal-family decomposition also gives type~\eqref{it:dim6 2-1,2-1,2-1} of Theorem~\ref{thm:dim6}.

\subsection{Dimension 3} The analysis for the six orientable flat 3-manifolds contained in this section was previously known \cite{Sell}. In this paper we also include the four non-orientable flat 3-manifolds. It turns out that every non-orientable flat 3-manifold appears as a cusp cross-section in every commensurability class of arithmetic hyperbolic 4-manifolds.

The commensurability classes of arithmetic hyperbolic 4-manifolds are defined by rational quadratic forms of signature $(4,1)$. As stated in Proposition~\ref{prop:projective_invariants}, the projective equivalence class of such a form $q$ is determined by the Hasse-Witt invariants $\varepsilon_p(q)$ for all $p$.

\begin{theorem}\label{thm:dim3}
The flat $3$-manifolds partition into three types according to their occurrence as cusp cross-sections in arithmetic commensurability classes. Each manifold belongs to exactly one of the following types, where the condition describes necessary and sufficient conditions for occurrence as a cusp cross-section in the commensurability class defined by $q$:
\begin{enumerate}[(a)]
    \item\label{it:dim3none} no conditions;
    \item\label{it:dim3mod3} $\varepsilon_p(q)=1$ for every prime $p \equiv 1 \bmod{3}$;
    \item\label{it:dim3mod4} $\varepsilon_p(q)=1$ for every prime $p \equiv 1 \bmod{4}$.
\end{enumerate}
The number of manifolds of each type is listed in Table~\ref{tab:dim3stats} with a breakdown of the distribution by holonomy group in Table~\ref{tab:dim3}.
\end{theorem}
\begin{table}[ht]
\begin{tabular}{|c | c | c | c|}
\hline
Type & $\mathcal{C}$ &  \# Manifolds & \# Orientable \\
\hline
\eqref{it:dim3none} & 1,1,1 & 7 & 3 \\
\hline
\eqref{it:dim3mod3}& 2-2,1 & 2 & 2 \\
\hline
\eqref{it:dim3mod4}& 2-1,1 & 1 & 1 \\
\hline \hline

Total &  & 10 & 6 \\
\hline
\end{tabular}
\caption{Distribution of flat 3-manifolds among the different types.}
\label{tab:dim3stats}
\end{table}

\begin{itemize}
    \item Every flat 3-manifold arises as a cusp cross-section in the arithmetic commensurability class defined by the form
    \[
        q_0(x)=x_1^2+x_2^2+x_3^2+x_4^2-x_5^2
    \]
    since $q_0$ satisfies $\varepsilon_p(q_0)=1$ for all primes $p$. 
    
    \item Every flat 3-manifold arises as a cusp cross-section in infinitely many distinct arithmetic commensurability classes. This is because there are infinitely many distinct projective equivalence classes of quadratic forms satisfying conditions \eqref{it:dim3none}, \eqref{it:dim3mod3} and \eqref{it:dim3mod4}. For example, one may consider the form 
    \[
    q_\ell(x)= \ell x_1^2+\ell x_2^2+x_3^2+x_4^2-x_5^2,
    \]
    for $\ell\equiv 11\bmod 12$ a prime. The projective invariants of $q_\ell$ can be computed as 
    \[
    \varepsilon_p(q_\ell)=(\ell,\ell)_p
    =\begin{cases}
        -1 & p=2, \ell\\
        1 & \text{otherwise.}
    \end{cases}
    \]
\end{itemize}

\subsection{Dimension 4}
The analysis for the 27 orientable flat 4-manifolds contained in this section was previously known \cite{McCoySell2024}. In this paper we extend the analysis to include the 47 non-orientable flat 4-manifolds too.
The commensurability classes of arithmetic hyperbolic 5-manifolds are defined by rational quadratic forms of signature $(5,1)$. The projective classes of such forms are determined by the discriminant $d(q)<0$ and the Hasse-Witt invariants $\varepsilon_p(q)$ at primes such that $-d(q)$ is a square in $\Q_p$.

\begin{theorem}\label{thm:dim4}
The flat 4-manifolds partition into four types according to their occurrence as cusp cross-sections in arithmetic commensurability classes. Each manifold belongs to exactly one of the following types, where the condition describes necessary and sufficient conditions for occurrence as a cusp cross-section in the commensurability class defined by $q$:
\begin{enumerate}[(a)]
    \item\label{it:dim4none} no conditions;
    \item\label{it:dim4mod3} $\varepsilon_p(q)=1$ for every prime such that $p\equiv 1 \bmod 3$ and $-d(q)\in \Qpxsquare$;
    \item\label{it:dim4mod4} $\varepsilon_p(q)=1$ for every prime such that $p\equiv 1 \bmod 4$ and $-d(q)\in \Qpxsquare$;
     \item\label{it:dim4modother} $\varepsilon_p(q)=1$ for every prime such that $-d(q)\in \Qpxsquare$.
     \end{enumerate}
The number of manifolds of each type is listed in Table~\ref{tab:dim4stats} with a breakdown of the distribution by holonomy group in Table~\ref{tab:dim4}.
\end{theorem}

\begin{table}[ht]
    \centering
\begin{tabular}{|c | c | c | c|}
\hline
Type & $\mathcal{C}$ &  \# Manifolds & \# Orientable \\
\hline
\eqref{it:dim4none} & 1,1,1,1 & 44  & 12 \\
\hline
\eqref{it:dim4mod3}& 2-2,1,1 & 11 & 7 \\
\hline
\eqref{it:dim4mod4}& 2-1,1,1 & 16  & 6 \\
\hline

\eqref{it:dim4modother}& 3,1 & 3  & 2 \\
\hline
\hline
Total & & 74 & 27 \\
\hline
\end{tabular}
\caption{Distribution of flat 4-manifolds among the different types.}
\label{tab:dim4stats}
\end{table}
\begin{itemize}
    \item Every flat 4-manifold arises as a cusp cross-section in the arithmetic commensurability class defined by the form
    \[
    q_0(x)=x_1^2+x_2^2+x_3^2+x_4^2+x_5^2-x_6^2
    \]
    since $q_0$ satisfies $d(q_0)=-1$ and $\varepsilon_p(q_0)=1$ for all $p$. 

    \item Every flat 4-manifold arises as a cusp cross-section in infinitely many distinct arithmetic commensurability classes. This is because there are infinitely many distinct projective equivalence classes of quadratic forms satisfying conditions \eqref{it:dim4none}, \eqref{it:dim4mod3}, \eqref{it:dim4mod4} and \eqref{it:dim4modother}. For example, we may consider
    \[
    q_\ell(x)= \ell x_1^2+x_2^2+x_3^2+x_4^2+x_5^2-x_6^2,
    \]
    for $\ell\equiv 1\bmod 4$ a prime. We have $d(q_\ell)=-\ell$ and $\varepsilon_p(q_\ell)=1$ for all $p$.
\end{itemize}

\subsection{Dimension 5}
The commensurability classes of arithmetic hyperbolic 6-manifolds are defined by rational quadratic forms of signature $(6,1)$. The projective classes of such forms are determined by the rescaled Hasse-Witt invariants $(d(q),-1)_p\varepsilon_p(q)$ at all primes.

\begin{theorem}\label{thm:dim5}
The flat 5-manifolds partition into seven types according to their occurrence as cusp cross-sections in arithmetic commensurability classes. Each manifold belongs to exactly one of the following types, where the condition describes necessary and sufficient conditions for occurrence as a cusp cross-section in the commensurability class defined by $q$:
\begin{enumerate}[(a)]
    \item\label{it:dim5none} no conditions;
    \item\label{it:dim5 2-2,2-2} $(d(q),-1)_p\varepsilon_p(q)=1$ for all $p\equiv 1 \bmod{3}$;
    \item\label{it:dim5 2-1,2-1} $(d(q),-1)_p\varepsilon_p(q)=1$ for all $p\equiv 1 \bmod{4}$;
    
    \item\label{it:dim5 2-1,2-2} $(d(q),-1)_p\varepsilon_p(q)=1$ for all $p\equiv 1\bmod{12}$;
    \item\label{it:dim5 4-3p} $(d(q),-1)_p\varepsilon_p(q)=1$ for all $p\equiv 1 \bmod{5}$;
    \item\label{it:dim5 4-1p} $(d(q),-1)_p\varepsilon_p(q)=1$ for all $p\equiv 1,3,5 \bmod{8}$;
    \item\label{it:dim5 4-2p} $(d(q),-1)_p\varepsilon_p(q)=1$ for all $p\equiv 1,5,7 \bmod{12}$.
 \end{enumerate}
The number of manifolds of each type is listed in Table~\ref{tab:dim5stats} with a breakdown of the distribution by holonomy group in Table~\ref{tab:dim5}.
\end{theorem}

\begin{table}[ht]
    \centering
\begin{tabular}{|c | c | c | c|}
\hline
Type & $\mathcal{C}$ &  \# Manifolds & \# Orientable \\
\hline
\eqref{it:dim5none} & \begin{tabular}{c}
     1,1,1,1,1\\ 2-1,1,1,1 \\ 2-2,1,1,1\\ 3,1,1
\end{tabular}& 1023 & 153 \\
\hline
\eqref{it:dim5 2-2,2-2}& 2-2,2-2,1 & 12 & 8 \\
\hline
\eqref{it:dim5 2-1,2-1}& 2-1,2-1,1 & 18 & 6 \\
\hline
\eqref{it:dim5 2-1,2-2}& 2-1,2-2,1 & 3 & 3 \\
\hline
\eqref{it:dim5 4-3p}& 4-3',1 & 2 & 2 \\
\hline
\eqref{it:dim5 4-1p}& 4-1',1 & 1 & 1 \\
\hline
\eqref{it:dim5 4-2p}& 4-2',1 & 1 & 1 \\
\hline
\hline
 Total & & 1060 & 174\\
 \hline
\end{tabular}
\caption{Distribution of flat 5-manifolds among the different types.}
\label{tab:dim5stats}
\end{table}

\begin{itemize}
    \item Every flat 5-manifold arises as a cusp cross-section in the arithmetic commensurability class defined by the form
    \[
        q_0(x)=x_1^2+ x_2^2+x_3^2+x_4^2+x_5^2 +x_6^2-x_7^2
    \]
    since $q_0$ satisfies $(d(q_0),-1)_p\varepsilon_p(q_0)=(-1,-1)_p$ for all $p$. In particular $(d(q_0),-1)_p\varepsilon_p(q_0)=1$ for every odd $p$. 
    \item Every flat 5-manifold arises as a cusp cross-section in infinitely many distinct arithmetic commensurability classes. This is because there are infinitely many distinct projective equivalence classes of quadratic forms satisfying the seven conditions. For example, we may consider
    \[
    q_\ell(x)= \ell x_1^2+\ell x_2^2+x_3^2+x_4^2+x_5^2 +x_6^2-x_7^2,
    \]
    for a prime $\ell\equiv 119\bmod 120$. For every prime $p$ we have
    \[
    (d(q_\ell),-1)_p\varepsilon_p(q_\ell)=(-1,-1)_p(\ell,\ell)_p=
        \begin{cases}
        -1 & p=\ell\\
        1 & \text{otherwise}
        \end{cases}.
        \]
\end{itemize}

\subsection{Dimension 6}

The commensurability classes of arithmetic hyperbolic 7-manifolds are defined by rational quadratic forms of signature $(7,1)$. The projective classes of such forms are determined by the discriminant $d(q)<0$ and the Hasse-Witt invariants $\varepsilon_p(q)$ at primes such that $d(q)$ is a square in $\Q_p$.

\begin{theorem}\label{thm:dim6}
The flat 6-manifolds partition into 10 types according to their occurrence as cusp cross-sections in arithmetic commensurability classes. Each manifold belongs to exactly one of the following types, where the condition describes necessary and sufficient conditions for occurrence as a cusp cross-section in the commensurability class defined by $q$:
 \begin{enumerate}[(a)]
 \item\label{it:dim6none} no conditions;
    \item\label{it:dim6 2-1,2-1} $\varepsilon_p(q)=1$ for all $p$ such that $p\equiv 1\bmod 4$ and $d(q)\in \Qpxsquare$;
    \item\label{it:dim6 2-2,2-2} $\varepsilon_p(q)=1$ for all $p$ such that $p\equiv 1\bmod 3$ and $d(q)\in \Qpxsquare$;
    \item\label{it:dim6 2-1,2-2} $\varepsilon_p(q)=1$ for all $p$ such that $p\equiv 1\bmod 12$ and $d(q)\in \Qpxsquare$;
    \item\label{it:dim6 4-1} $\varepsilon_p(q)=(-1,-1)_p$ for all $p$ such that $d(q)\in \Qpxsquare$;
    \item\label{it:dim6 4-3p} $\varepsilon_p(q)=1$ for all $p$ such that $p\equiv 1 \bmod{5}$ and $d(q)\in \Qpxsquare$;
    \item\label{it:dim6 4-1p} $\varepsilon_p(q)=1$ for all $p$ such that $p\equiv 1,3,5 \bmod{8}$ and $d(q)\in \Qpxsquare$;
    \item\label{it:dim6 4-2p} $\varepsilon_p(q)=1$ for all $p$ such that $p\equiv 1,5,7 \bmod{12}$ and $d(q)\in \Qpxsquare$;
    \item\label{it:dim6 4-2} $\varepsilon_p(q)=(-3,-1)_p$ for all $p$ such that $d(q)\in \Qpxsquare$;
    \item\label{it:dim6 2-1,2-1,2-1} $d(q)=-1$ and $\varepsilon_p(q)=1$ for all $p$ such that $d(q)\in \Qpxsquare$.
\end{enumerate}
The number of manifolds of each type is listed in Table~\ref{tab:dim6stats} with a breakdown of the distribution by holonomy group in Table~\ref{tab:dim6}.
\end{theorem}

\begin{longtable}{|c|c | c | c|}
\hline
Type & $\mathcal{C}$ &  \# Manifolds & \# Orientable \\
\hline
\eqref{it:dim6none} &\begin{tabular}{c}
     1,1,1,1,1,1\\ 2-1,1,1,1,1 \\ 2-2,1,1,1,1\\3,1,1,1
\end{tabular} & 37216 & 2868 \\
\hline
\eqref{it:dim6 2-1,2-1}& \begin{tabular}{c}
     2-1,2-1,1,1\\ 3,2-1,1
\end{tabular} & 1088 & 255 \\
\hline
\eqref{it:dim6 2-2,2-2}& \begin{tabular}{c}
     2-2,2-2,1,1\\ 3,2-2,1
\end{tabular} & 199 & 106 \\
\hline
\eqref{it:dim6 2-1,2-2}&2-1,2-2,1,1 & 160 & 75 \\
\hline
\eqref{it:dim6 4-1}& \begin{tabular}{c}
     4-1,1,1\\ 5-1,1
\end{tabular}& 34 & 4 \\
\hline
\eqref{it:dim6 4-3p}&4-3',1,1 & 11 & 3 \\
\hline
\eqref{it:dim6 4-1p}& 4-1',1,1 & 14 & 2 \\
\hline
\eqref{it:dim6 4-2p}& 4-2',1,1 & 4 & 1 \\
\hline
\eqref{it:dim6 4-2}& 4-2,1,1 & 12 & 0 \\
\hline
\eqref{it:dim6 2-1,2-1,2-1}&\begin{tabular}{c}
     2-1,2-1,2-1\\ 4-1,2-1
\end{tabular}  & 8 & 0 \\
\hline
\hline
 Total & & 38746 & 3314\\
 \hline
\caption{Distribution of flat 6-manifolds among the different types.}
\label{tab:dim6stats}
\end{longtable}

\begin{itemize}
    \item Every flat 6-manifold arises as a cusp cross-section in the arithmetic commensurability class defined by the form
    \[
        q_0(x)=x_1^2+ x_2^2+x_3^2+x_4^2+x_5^2 +x_6^2+x_7^2-x_8^2
    \]
    since $q_0$ satisfies $d(q_0)=-1$ and $\varepsilon_p(q_0)=1$ for all $p$. Note that this is consistent with types \eqref{it:dim6 4-1} and \eqref{it:dim6 4-2}, since if $d(q_0)=-1$ is a square in $\Q_p$, then $p\equiv 1\bmod 4$ implying that $(-1,-1)_p=(-3,-1)_p=1$. 
    \item The flat 6-manifolds which are not of type \eqref{it:dim6 2-1,2-1,2-1} all arise as cusp cross-sections in infinitely many distinct arithmetic commensurability classes. This is because there are infinitely many distinct projective equivalence classes of quadratic forms satisfying the remaining nine conditions. This can be seen by considering
    \[
    q_\ell(x)= \ell x_1^2+x_2^2+x_3^2+x_4^2+x_5^2+x_6^2+x_7^2-x_8^2,
    \]
    for $\ell\equiv 1\bmod 12$ a prime. We have $d(q_\ell)=-\ell$ and, since $\ell\equiv 1\bmod 4$, that $\varepsilon_p(q_\ell)=1$ for all $p$. This immediately shows that $q_\ell$ satisfies all the conditions except possibly \eqref{it:dim6 4-1} and \eqref{it:dim6 4-2}. To see that it also satisfies \eqref{it:dim6 4-1} and \eqref{it:dim6 4-2}, observe that $d(q_\ell)=-\ell\equiv 2 \bmod 3$ and $d(q_\ell)=-\ell\equiv 3 \bmod 4$. This implies that $d(q_\ell)$ is not a square in $\Q_3$ or $\Q_2$. Thus if $d(q_\ell)\in \Qpxsquare$, then $(-1,-1)_p=(-3,-1)_p=1$. Thus $q_\ell$ is also consistent with \eqref{it:dim6 4-1} and \eqref{it:dim6 4-2}.
    \item The eight flat manifolds of type \eqref{it:dim6 2-1,2-1,2-1} have the UCC property. By Proposition~\ref{prop:projective_invariants}, the conditions of type \eqref{it:dim6 2-1,2-1,2-1} specify a unique projective equivalence class of quadratic forms and hence determine a unique commensurability class of arithmetic manifolds. Note that there are no orientable manifolds of type \eqref{it:dim6 2-1,2-1,2-1}. 
    \item A flat 6-manifold $C$ with the UCC property was constructed in \cite[Lemma~4.2]{McCoySell2025}. By construction, $C$ has holonomy group of order 16. It follows from Table~\ref{tab:dim6} that $C$ corresponds to one of the four flat manifolds of type~\eqref{it:dim6 2-1,2-1,2-1} with holonomy group $H\cong(C_4 \times C_2) \rtimes C_2$.  
\end{itemize}

\subsection{Proofs of the main theorems}
We conclude this section by collecting the preceding results to prove Theorems~\ref{thm:UCC_analysis} and~\ref{thm:universal_class}.

\begin{proof}[Proof of Theorem~\ref{thm:universal_class}]
Let $B$ be a compact flat $n$-manifold, where $3 \leq n \leq 6$, and consider the quadratic form
\[
q_0(x)=x_1^2+\cdots+x_{n+1}^2-x_{n+2}^2.
\]
We claim that $B$ occurs as a cusp cross-section in the arithmetic commensurability class defined by $q_0$.

For $n=3$, the form $q_0$ satisfies
\[
\varepsilon_p(q_0)=1
\]
for every prime $p$, and hence satisfies each of the conditions in Theorem~\ref{thm:dim3}. For $n=4$, we have
\[
d(q_0)=-1
\qquad\text{and}\qquad
\varepsilon_p(q_0)=1
\]
for every prime $p$, so $q_0$ satisfies each of the conditions in Theorem~\ref{thm:dim4}.

For $n=5$, we have
\[
(d(q_0),-1)_p\varepsilon_p(q_0)=(-1,-1)_p.
\]
In particular,
\[
(d(q_0),-1)_p\varepsilon_p(q_0)=1
\]
for every odd prime $p$, and consequently $q_0$ satisfies each of the conditions in Theorem~\ref{thm:dim5}.

Finally, suppose that $n=6$. Again
\[
d(q_0)=-1
\qquad\text{and}\qquad
\varepsilon_p(q_0)=1
\]
for every prime $p$. Thus $q_0$ clearly satisfies the conditions for all of the types in Theorem~\ref{thm:dim6} with the possible exceptions of \eqref{it:dim6 4-1} and \eqref{it:dim6 4-2}. For types~\eqref{it:dim6 4-1} and~\eqref{it:dim6 4-2}, observe that if $d(q_0)=-1$ is a square in $\mathbb{Q}_p$, then $p\equiv 1\bmod 4$, and hence
\[
(-1,-1)_p=(-3,-1)_p=1=\varepsilon_p(q_0).
\]
Thus $q_0$ also satisfies the conditions of types~\eqref{it:dim6 4-1} and~\eqref{it:dim6 4-2}.

It follows from Theorems~\ref{thm:dim3}--\ref{thm:dim6} that every compact flat $n$-manifold, for $3\leq n\leq 6$, occurs as a cusp cross-section in the commensurability class defined by $q_0$.
\end{proof}

\begin{proof}[Proof of Theorem~\ref{thm:UCC_analysis}]
For dimensions $n=3,4,5$, Theorems~\ref{thm:dim3}--\ref{thm:dim5} give the complete list of conditions imposed on the projective equivalence class of a quadratic form by a flat $n$-manifold. In each case, the families of quadratic forms exhibited above show that every one of these conditions is satisfied by infinitely many distinct projective equivalence classes. Consequently every flat $n$-manifold, for $n=3,4,5$, occurs as a cusp cross-section in infinitely many distinct commensurability classes of arithmetic hyperbolic $(n+1)$-manifolds.

Now suppose that $n=6$. By Theorem~\ref{thm:dim6}, every flat $6$-manifold is of one of the types \eqref{it:dim6none}--\eqref{it:dim6 2-1,2-1,2-1}. For each of the types \eqref{it:dim6none}--\eqref{it:dim6 4-2}, the forms
\[
q_\ell(x)
=
\ell x_1^2+x_2^2+\cdots+x_7^2-x_8^2,
\]
where $\ell\equiv 1\bmod{12}$ is prime, satisfy the corresponding condition. Since
\[
d(q_\ell)=-\ell,
\]
varying $\ell$ gives infinitely many distinct projective equivalence classes. Hence every flat $6$-manifold of type \eqref{it:dim6none}--\eqref{it:dim6 4-2} occurs as a cusp cross-section in infinitely many distinct arithmetic commensurability classes.

It remains to consider type \eqref{it:dim6 2-1,2-1,2-1}. In this case the conditions are
\[
d(q)=-1
\qquad\text{and}\qquad
\varepsilon_p(q)=1
\quad\text{whenever}\quad
-1\in(\mathbb{Q}_p^\times)^2.
\]
By Proposition~\ref{prop:projective_invariants}, these conditions determine a unique projective equivalence class of rational quadratic forms of signature $(7,1)$. Thus every flat $6$-manifold of type \eqref{it:dim6 2-1,2-1,2-1} has the UCC property. By Table~\ref{tab:dim6stats} there are exactly eight such manifolds, and none of them are
orientable.

Therefore, with precisely eight non-orientable exceptions, every flat $6$-manifold occurs as a cusp cross-section in infinitely many commensurability classes of arithmetic hyperbolic $7$-manifolds, and the eight exceptional manifolds have the UCC property.
\end{proof}

\bibliographystyle{alpha}
\bibliography{bib}

\appendix

\section{Crystal families: $\R$-irreducibles}\label{sec:R_irred}
For each of the crystal families that contain $\R$-irreducible representations, we classify the positive definite quadratic forms whose orthogonal groups contain representations of that crystal family. Since an $\R$-irreducible representation admits a 1-dimensional space of invariant forms, there will be a unique projective class of such forms. Thus for each crystal family it suffices to take the projective equivalence of any positive-definite rational quadratic form invariant under some representation in that crystal family.
\subsection{The family 3}
The crystal family $3$ is the family containing any 3-dimensional $\Q$-irreducible representation. One example of such a representation is has image generated by the matrices
\[\text{
$A=
\begin{pmatrix}
0 & 0 & 1\\
1 & 0 & 0\\
0 & 1 & 0
\end{pmatrix}$
and
$B=
\begin{pmatrix}
1 & 0 & 0\\
0 & -1 & 0\\
0 & 0 & -1
\end{pmatrix}.$}
\]
Since these matrices are orthogonal and the representation is $\R$-irreducible, a symmetric matrix $M$ satisfies
\[
A^T M A = B^T M B =M
\]
if and only if it takes the form
\[
M=\begin{pmatrix}
        a&0 &0\\
        0&a&0\\
        0& 0&a
    \end{pmatrix}
\]
for some $a\in\Q$.
Lemma~\ref{lem:dim3analysis} follows from this calculation.

\subsection{The family 4-1}
The family $4-1$ is the family containing the faithful rational representation of the quaternion group $Q_8$. One example of such a representation one generated by the matrices
 \[
    A=\begin{pmatrix}
        0&-1 &0&0\\
        1&0&0&0\\
        0& 0&0&-1\\
        0& 0&1 &0
    \end{pmatrix}\quad\text{and}\quad
    B=\begin{pmatrix}
        0&0 &1&0\\
        0&0&0&-1\\
        -1& 0&0&0\\
        0& 1 &0&0
    \end{pmatrix}.
    \]
Since these matrices are orthogonal and the representation is $\R$-irreducible, a symmetric rational matrix $M$ satisfies
\[
A^T M A = B^T M B =M
\]
if and only if it takes the form
\[
M=\begin{pmatrix}
        a&0 &0&0\\
        0&a&0&0\\
        0& 0&a&0\\
        0 &0&0&a
    \end{pmatrix}
\]
for some $a\in \Q$.
Lemma~\ref{lem:dim4analysis}\eqref{it:4-1} follows from this calculation.

\subsection{The family 4-2}
This is the family containing an irreducible rational representation of $C_3\rtimes C_4$. 
There is such a representation of $C_3\rtimes C_4$ whose image is generated by the matrices
\[\text{$A=
\begin{pmatrix}
0&-1&0&0\\
1&-1&0&0\\
0&0&-1&1\\
0&0&-1&0
 \end{pmatrix}$ and $B=
 \begin{pmatrix}
0&0&-1&0\\
0&0&0&-1\\
1&0&0&0\\
0&1&0&0
 \end{pmatrix}$}.
 \]
One can check that a rational symmetric matrix $M$ satisfies $A^TMA=B^TMB=M$ if and only if it takes the form
\[
M=
\begin{pmatrix}
    2a&a&0&0\\ a& 2a& 0& 0\\ 0& 0& 2a& a\\ 0& 0 &a &2a
\end{pmatrix}
\]
for some $a\in \Q$. Lemma~\ref{lem:dim4analysis}\eqref{it:4-2} follows since it is easy to check that a positive definite form is equivalent to one represented by such a matrix if and only if it satisfies $d(f)=1$ and $\varepsilon_p(f)=(3,3)_p$ for all $p$. In particular, a form represented by such a matrix $M$ can be diagonalized to a form represented by a matrix
\[
M'=
\begin{pmatrix}
    3a&0&0&0\\ 0& a& 0& 0\\ 0& 0& 3a& 0\\ 0& 0 &0 &a
\end{pmatrix}.
\]
\subsection{The family 4-3}
This is the crystal family containing an irreducible 4-dimensional representation of $A_5$. There is such a representation with image generated by the matrices 
\[\text{
$A_1=
\begin{pmatrix}
-1 & -1 & -1 & -1\\
1  & 0  & 0  & 0\\
0  & 1  & 0  & 0\\
0  & 0  & 1  & 0
\end{pmatrix}$ 
and
$A_2=
\begin{pmatrix}
0 & 1 & 0 & 0\\
1 & 0 & 0 & 0\\
0 & 0 & 0 & 1\\
0 & 0 & 1 & 0
\end{pmatrix}$.}
\]
One can check that a rational symmetric matrix $M$ satisfies $(A_1)^TMA_1=(A_2)^TMA_2=M$ if and only if it takes the form
\[
M=
\begin{pmatrix}
2a & a & a & a\\
a & 2a & a & a\\
a & a & 2a & a\\
a & a & a & 2a
\end{pmatrix}
\]
for some $a\in \Q$. Lemma~\ref{lem:dim4analysis}\eqref{it:4-3} follows since it is easy to check that a positive definite form is equivalent to one represented by such a matrix if and only if it satisfies $d(f)=5$ and $\varepsilon_p(f)=1$ for all $p$ such that $5\in \Qpxsquare$. However 5 is a square in $\Q_p$ if and only if $p\equiv 1,4\bmod 5$.
\subsection{The family 5-1}
This is the crystal family containing a faithful 5-dimensional representation of $C_2^5 \rtimes S_5$. Such a representation can clearly be chosen to be orthogonal and since it is $\R$-irreducible, we see that every form which is invariant under such a representation satisfies $\varepsilon_p(f)=1$ for all $p$. This proves Lemma~\ref{lem:dim5analysis}\eqref{it:5-1}.

\subsection{The family 5-2}
This is the crystal family containing the irreducible 5-dimensional representation of $A_6$. There is such a representation with image generated by the matrices
\[
\text{
$A=
\begin{pmatrix}
0 & 0 & 1 & 0 & 0\\
1 & 0 & 0 & 0 & 0\\
0 & 1 & 0 & 0 & 0\\
0 & 0 & 0 & 1 & 0\\
0 & 0 & 0 & 0 & 1
\end{pmatrix}$
and
$B=
\begin{pmatrix}
1 & 0 & 0 & 0 & 0\\
-1 & -1 & -1 & -1 & -1\\
0 & 1 & 0 & 0 & 0\\
0 & 0 & 1 & 0 & 0\\
0 & 0 & 0 & 1 & 0
\end{pmatrix}$,}
\]
where $A$ is the image of the 3-cycle $(123)$ and $B$ is the image of the 5-cycle $(23456)$.

One can check that a rational symmetric matrix $M$ satisfies $A^TMA=B^TMB=M$ if and only if it takes the form
\[
M=
\begin{pmatrix}
2a & a & a & a & a\\
a & 2a & a & a & a\\
a & a & 2a & a & a\\
a & a & a & 2a & a\\
a & a & a & a & 2a
\end{pmatrix},
\]
for some $a\in \Q$.
The projective equivalence class of this form is characterized by the fact that $\varepsilon_p(f)=(3,3)_p$ for all primes $p$. This proves Lemma~\ref{lem:dim5analysis}\eqref{it:5-2}.

\section{Crystal families: cyclic groups}\label{sec:R_red}
We now consider each of the crystal families 4-1', 4-2' and 4-3' that do not contain $\R$-irreducible representations. Each of these families contains representations of cyclic groups, so we will be able to make use of the following result.
\begin{prop}[{\cite[Proposition~3.1]{McCoySell2025}}]\label{prop:cyclic_rep_calc}
    Let 
    \[\sigma : \Z/n \rightarrow O(f;\Q)\]
    be a faithful irreducible representation, where $f$ is a positive definite rational quadratic form and $n>2$. Then
    \begin{enumerate}[(i)]
        \item $f$ has rank $\varphi(n)$
        \item the discriminant of $f$ satisfies \[
    d(f)=\begin{cases}
        p &\text{if $n=p^r$ or $n=2p^r$ for $r\geq 1$ and $p$ an odd prime}\\
        1 &\text{otherwise.}
    \end{cases}
    \]
    \item If there exists an integer $m\geq 1$ such that $-m$ is a square in both the cyclotomic field $\Q[\zeta_n]$ and the $q$-adic field $\Q_q$ for some prime $q$, then the Hasse-Witt invariant satisfies
    \[\varepsilon_q(f)=\begin{cases} 1&q>2,\\
(-1)^{\frac{\varphi(n)(\varphi(n)-2)}{8}} &q=2.
\end{cases}
\]
\qed
\end{enumerate}  
\end{prop}
A key technical ingredient will be the following theorem that characterizes the rationals that can be represented by a given rank two rational quadratic form.
\begin{theorem}[{\cite{Serre_arithmetic}}]\label{thm:rank2_realization}
    Let $Q$ be a non-degenerate 2-dimensional rational quadratic form. Then for $c\in \Q^\times$, the equation $Q(x)=c$ has a solution if and only if
    \[
    (c, -d(Q))_p=\varepsilon_p(Q)
    \]
    for all valuations.\qed
\end{theorem}
\subsection{The family 4-1'}
The family 4-1' is the crystal family containing the faithful rational representation of $C_8$. Let $f$ be a positive definite rational quadratic form such that $O(f;\Q)$ contains the image of a representation in the family 4-1'. From Proposition~\ref{prop:cyclic_rep_calc}, we see that $d(f)=1$.
The imaginary quadratic subfields of $\Q[\zeta_8]$ are $\Q[\sqrt{-1}]$ and $\Q[\sqrt{-2}]$. Thus we have that $\varepsilon_p(f)=1$ if $-2$ or $-1$ is a square in $\Q_p$. However $-1$ is a square in $\Q_p$ if and only if $p\equiv 1\bmod 4$ and $-2$ is a square in $\Q_p$ if and only if $p\equiv 1, 3 \bmod 8$. Thus, Proposition~\ref{prop:cyclic_rep_calc} shows that $\varepsilon_p(f)=1$ whenever $p\equiv 1,3,5\bmod 8$. Thus Proposition~\ref{prop:cyclic_rep_calc} gives all the restrictions on $f$ required by Lemma~\ref{lem:dim4analysis}\eqref{it:4-1'}. It remains to verify that all forms satisfying these restrictions can be realized.

The aim is to show that for any set of distinct primes $q_1,\dots, q_\ell$ such that $q_i\equiv 7 \bmod 8$, there exists a positive definite quadratic form $f$ of rank 4 such that $O(f;\Q)$ contains a matrix of order 8 and we have
\[
\varepsilon_p(f)=\begin{cases}
    -1 &\text{if $p\in \{q_1, \dots, q_\ell\}$}\\
    (-1)^{\ell} &\text{if $p=2$}\\
    1 &\text{otherwise}.
\end{cases}
\]
Let $f$ be a rational quadratic form defined by the matrix 
\[
M=\begin{pmatrix}
    a&b&0&-b\\
    b&a&b&0\\
    0&b&a&b\\
    -b&0&b&a
\end{pmatrix},
\]
where $a,b\in\Q$ are not both zero.

Since the matrix
\[
A=\begin{pmatrix}
    0&0&0&-1\\
    1&0&0&0\\
    0&1&0&0\\
    0&0&1&0
\end{pmatrix}
\]
has order 8 and satisfies $A^T M A =M$, we see that there is a representation of family 4-1' with image in $O(f;\Q)$. 
\begin{claim}
    If $a>0$ and $a^2-2b^2>0$, then $f$ is positive definite and for any prime $p$, we have
    \[
    \varepsilon_p(f)=(a^2-2b^2,-1)_p.
    \]
\end{claim}
\begin{proof}[Proof of Claim]
The leading principal minors of $M$ are
\[
\begin{aligned}
\Delta_1 &= a,\\
\Delta_2 &= a^2-b^2,\\
\Delta_3 &= a(a^2-2b^2),\\
\Delta_4 &= (a^2-2b^2)^2.
\end{aligned}
\]
If $a>0$ and $a^2-2b^2>0$, then these minors are all positive and so $f$ can be diagonalized as
\[
f'=\left\langle a, \frac{a^2-b^2}{a}, \frac{a(a^2-2b^2)}{a^2-b^2}, \frac{(a^2-2b^2)}{a}  \right\rangle.
\]
Since $f'$ is positive definite this implies that $f$ itself is positive definite. Since we may multiply entries by squares without changing the rational equivalence class, we see that this is further equivalent to
\[
f''=\left\langle a, a\alpha, a\alpha\beta, a\beta   \right\rangle.
\]
where $\alpha =a^2-b^2$ and $\beta=a^2-2b^2$. Since $\alpha -\beta =b^2$, we have that $(\alpha,-\beta)_p=1$ for all primes $p$. Thus we can calculate the Hasse-Witt invariants of $f$ as
\begin{align*}
    \varepsilon_p(f)&=(\alpha,\alpha\beta)_p(\alpha,\beta)_p(\alpha\beta,\beta)_p\\
    &=(\alpha,-\beta)_p(\alpha,\beta)_p(-\alpha,\beta)_p\\
    &=(\beta,-1)_p.\\
\end{align*}
\end{proof}

Consider the rational quadratic form given by $\beta(x,y)=x^2-2y^2$. This satisfies $d(\beta)=-2$ and $\varepsilon_p(\beta)=1$ for all primes $p$. Let $q_1,\dots, q_\ell$ be a finite collection of primes $q_i\equiv 7 \bmod 8$. For any such $q_i$, we have $(q_i,2)_p=1$ for all $p$. Thus Theorem~\ref{thm:rank2_realization} implies that we can find $a,b\in \Q_{>0}$ such that $a^2-2b^2=q_1\cdots q_\ell$. Thus if we take $f$ to be a form defined by such an $a$ and $b$, then the claim implies that $f$ is positive definite and that

\begin{align*}
    \varepsilon_p(f)&=(q_1\cdots q_\ell, -1)_p\\
    &=\begin{cases}
    -1 &\text{if $p\in \{q_1, \dots, q_\ell\}$}\\
    (-1)^{\ell} &\text{if $p=2$}\\
    1 &\text{otherwise}.
\end{cases}
\end{align*}
This completes the verification of Lemma~\ref{lem:dim4analysis}\eqref{it:4-1'}.

\subsection{The family 4-2'}
The family 4-2' is the crystal family containing the faithful rational representation of $C_{12}$.
Let $f$ be a positive definite rational quadratic form such that $O(f;\Q)$ contains the image of a representation in the family 4-2'. From Proposition~\ref{prop:cyclic_rep_calc}, we see that $d(f)=1$.
The imaginary quadratic subfields of $\Q[\zeta_{12}]$ are $\Q[\sqrt{-1}]$ and $\Q[\sqrt{-3}]$. Thus we have that $\varepsilon_p(f)=1$ if $-3$ or $-1$ is a square in $\Q_p$. However $-1$ is a square in $\Q_p$ if and only if $p\equiv 1\bmod 4$ and $-3$ is a square in $\Q_p$ if and only if $p\equiv 1 \bmod 3$. Thus, Proposition~\ref{prop:cyclic_rep_calc} shows that $\varepsilon_p(f)=1$ whenever $p\equiv 1,5,7\bmod 12$. Thus Proposition~\ref{prop:cyclic_rep_calc} gives all the restrictions on $f$ required by Lemma~\ref{lem:dim4analysis}\eqref{it:4-2'}. It remains to verify that all forms satisfying these restrictions can be realized.

The aim is to show that for any collection of distinct primes $q_1, \dots, q_\ell$ such that $q_i\equiv 3,11 \bmod{12}$, there exists a positive definite quadratic form $f$ of rank 4 such that $O(f;\Q)$ contains a matrix of order 12 and we have
\[
\varepsilon_p(f)=\begin{cases}
    -1 &\text{if $p\in \{q_1, \dots, q_\ell\}$}\\
    (-1)^{\ell} &\text{if $p=2$}\\
    1 &\text{otherwise}.
\end{cases}
\]

Let $f$ be a quadratic form defined by the symmetric matrix
\[
M=
\begin{pmatrix}
    2a&b&a&0\\ b& 2a& b& a\\ a& b& 2a& b\\ 0& a &b &2a
\end{pmatrix}
\]
Since the matrix
    \[
A=\begin{pmatrix}
    0&0&0&-1\\
    1&0&0&0\\
    0&1&0&1\\
    0&0&1&0.
\end{pmatrix}
\]
has order 12 and satisfies $A^TMA=M$, we see that there is a representation in the family 4-2' taking values in $O(f;\Q)$.

\begin{claim}
    If $a>0$ and $3a^2-b^2>0$, then $f$ is positive definite and for any prime $p$ we have
    \[
    \varepsilon_p(f)=(3a^2-b^2,-1)_p.
    \]
\end{claim}
\begin{proof}[Proof of Claim]
The leading principal minors of $M$ are
\[
\begin{aligned}
\Delta_1&=2a,\\
\Delta_2&=4a^2-b^2,\\
\Delta_3&=2a(3a^2-b^2),\\
\Delta_4&=(3a^2-b^2)^2.
\end{aligned}
\]
Thus if $a>0$ and $3a^2-b^2>0$, then these minors are all positive and so $f$ can be diagonalized as the positive definite form
\[
f'=\left\langle 2a, \frac{4a^2-b^2}{2a}, \frac{2a(3a^2-b^2)}{4a^2-b^2}, \frac{(3a^2-b^2)}{2a}  \right\rangle.
\]
This is further equivalent to the form
\[
f''=\left\langle 2a, 2a\alpha, 2a\alpha\beta, 2a\beta  \right\rangle,
\]
where $\alpha=4a^2-b^2$ and $\beta=3a^2-b^2$. Since $\alpha-\beta=a^2$, we have $(\alpha,-\beta)_p=1$ for all primes $p$. Thus we may calculate the Hasse-Witt invariants of $f$ by
\begin{align*}
\varepsilon_p(f)&=(\alpha, \alpha \beta)_p(\alpha, \beta)_p(\alpha\beta, \beta)_p\\
&=(\alpha, -\beta)_p(\beta,-1)_p\\
&=(\beta,-1)_p.
\end{align*}
\end{proof}

Consider the rational quadratic form given by $\beta(x,y)=3x^2-y^2$. This satisfies $d(\beta)=-3$ and $\varepsilon_p(\beta)=(3,-1)_p$ for all primes $p$. Let $q$ be a prime $q=2$, $q=3$ or $q\equiv 11 \bmod 12$. For such a prime one can calculate that \[
(q,3)_p=(3,-1)_p=
\begin{cases}
    1 & p>3\\
    -1 & p=2,3
\end{cases}
\] 
for all $p$. Thus if we take $q_1,\dots, q_\ell$ to be a collection of distinct primes such that $q_i=3$ or $q_i\equiv 11 \bmod 12$, then
\[
(2^{\ell-1} q_1\dots q_\ell,3)_p=(3,-1)_p,
\]
for all $p$. Thus Theorem~\ref{thm:rank2_realization} implies that we can find $a,b\in \Q_{>0}$ such that $3a^2-b^2=2^{\ell-1}q_1\cdots q_\ell$. For such a choice of $a, b$, we have that $f$ is positive definite and one can calculate

\begin{align*}
\varepsilon_p(f)&=(2^{\ell-1}q_1\cdots q_\ell,-1)_p\\
&=\begin{cases}
    -1 &\text{if $p\in \{q_1, \dots, q_\ell\}$}\\
    (-1)^{\ell} &\text{if $p=2$}\\
    1 &\text{otherwise}.
\end{cases}
\end{align*}
This completes the verification of Lemma~\ref{lem:dim4analysis}\eqref{it:4-2'}.

\subsection{The family 4-3'}
The family 4-3' is the crystal family containing the faithful rational representation of $C_5$. Let $f$ be a positive definite rational quadratic form such that $O(f;\Q)$ contains the image of a representation in the family 4-3'. From Proposition~\ref{prop:cyclic_rep_calc}, we see that $d(f)=5$. On the other hand if $\Q_p$ contains a primitive 5th root of unity, we see that a representation of family 4-3' will split into 4 complex representations and we will have $\varepsilon_p(f)=1$. This implies that $\varepsilon_p(f)=1$ whenever $p\equiv 1\bmod 5$.

Since every form $f$ satisfies $d(f)=5$, we see that its projective equivalence class is determined by the Hasse-Witt invariants $\varepsilon_p(f)$ for all $p$ such that $5$ is a square in $\Q_p$. These primes are precisely those $p$ such that $p\equiv 1,4\bmod 5$. Thus it remains to show that for every collection of distinct primes $q_1, \dots, q_\ell$ such that $q_i\equiv 4\bmod 5$, there exists a positive definite quadratic form $f$ of rank 4 such that $O(f;\Q)$ contains a matrix of order 5 and
\[
\varepsilon_p(f)=\begin{cases}
    -1 &p=q_1,\dots, q_\ell\\
    1 &\text{$p\equiv 4\bmod 5$ and $p\neq q_1,\dots, q_\ell $}
\end{cases}
\]

Let $f$ be a quadratic form represented by a symmetric matrix
\[
M=\begin{pmatrix}
    2(a+b) & -a & -b&-b\\
    -a & 2(a+b) & -a & -b\\
    -b& -a & 2(a+b)& -a\\
    -b& -b & -a & 2(a+b)
\end{pmatrix}
\]
Since 
\[
A=\begin{pmatrix}
    0&0&0&-1\\
    1&0&0&-1\\
    0&1&0&-1\\
    0&0&1&-1.
\end{pmatrix}
\]
is a matrix of order 5 and $A^TMA=M$, we see that there is a representation of family 4-3' taking values in $O(f;\Q)$.

\begin{claim}
    Over the field $\Q[\sqrt{5}]$, $f$ is equivalent to the diagonal form
    \begin{equation}\label{eq:root5diag}
        f'''=\langle \alpha, d\alpha, \beta, d\beta \rangle,
    \end{equation}
where $\alpha=2a+3b+b\sqrt{5}$, $\beta=2a+3b-b\sqrt{5}$ and $d=\frac{5+\sqrt{5}}{8}.$
\end{claim}
\begin{proof}[Proof of Claim]
If we take $X\in GL_4(\Q[\sqrt{5}])$ to be the matrix
\[
X=\frac{1}{\sqrt{5}}\begin{pmatrix}
4&-1+\sqrt{5}&4&-1-\sqrt{5}\\
3+\sqrt{5}&3+\sqrt{5}&3-\sqrt{5}&3-\sqrt{5}\\
2& 2+2\sqrt{5}&2&2-2\sqrt{5}\\
1-\sqrt{5}&1+\sqrt{5}&1+\sqrt{5}&1-\sqrt{5},
\end{pmatrix}
\]
then we see that $f$ is equivalent over $\Q[\sqrt{5}]$ to the form represented by the matrix $M'=X^T M X$. One can calculate that
\[
M'=2
\begin{pmatrix}
    4a+6b+2b\sqrt{5} & b-a+(a+b)\sqrt{5} &0 &0 \\
    b-a+(a+b)\sqrt{5} &4a+6b+2b\sqrt{5} &0 &0 \\
   0 &0&4a+6b-2b\sqrt{5} & b-a-(a+b)\sqrt{5} \\
    0&0&b-a-(a+b)\sqrt{5} &4a+6b-2b\sqrt{5}.
\end{pmatrix}
\]
Conveniently $M'$ can be expressed in the form
\[M'=4
\begin{pmatrix}
    \alpha & c_1\alpha &0 &0 \\
    c_1\alpha &\alpha &0 &0 \\
    0&0&\beta & c_2\beta \\
    0&0&c_2\beta &\beta 
\end{pmatrix}
\]
where $\alpha=2a+3b+b\sqrt{5}$, $c_1=\frac{-1+\sqrt{5}}{4}$ and $\beta=2a+3b-b\sqrt{5}$, $c_2=\frac{-1-\sqrt{5}}{4}$. Thus over $\Q[\sqrt{5}]$, we see that $f$ can be diagonalized to the form represented by the matrix
\[
f''=\langle \alpha, (1-c_1^2)\alpha, \beta, (1-c_2^2)\beta \rangle,
\]

Since $1-c_1^2=\frac{5+\sqrt{5}}{8}$ and $1-c_2^2=\frac{5-\sqrt{5}}{8}$, one can check that $1-c_1^2=(\frac{1+\sqrt{5}}{2})^2 (1-c_2^2)$. Thus $f$ is equivalent over $\Q[\sqrt{5}]$ to the diagonal form
\[
f'''=\langle \alpha, \gamma\alpha, \beta, \gamma\beta \rangle,
\]
where $\gamma=1-c_1^2=\frac{5+\sqrt{5}}{8}$.
\end{proof}
This allows us to calculate the Hasse-Witt invariants of $f$ for $p\equiv 4\bmod 5$.
\begin{claim}
    Suppose that $a^2+3ab+b^2>0$. Then $f$ is definite. For any prime $p\equiv 4\bmod 5$, we have
    \[
    \varepsilon_p(f)=-1
    \]
   if and only if the $p$-adic valuation of $a^2+3ab+b^2$ is odd.
\end{claim}
\begin{proof}[Proof of Claim]
Since the diagonalization of \eqref{eq:root5diag} is valid over $\R$ and $d>0$, we see that $f$ is definite if and only if $\alpha\beta=4(a^2+3ab+b^2)>0$.

Now let $p$ be any prime such that $p\equiv 4\bmod 5$. Since $5$ is a quadratic residue mod $p$, we see that $5$ is a square in $\Q_p$. Thus we may choose an embedding of $\Q[\sqrt{5}]$ into $\Q_p$, allowing us to consider $\Q[\sqrt{5}]$ as a subfield of $\Q_p$.  

Since the diagonalization \eqref{eq:root5diag} of $f$ can be performed over $\Q_p$, we have
\begin{align*}
    \varepsilon_p(f)&=(\alpha \beta, -\gamma)_p (\gamma,\gamma)_p\\
    &= (a^2 + 3ab + b^2, -\gamma)_p (\gamma,\gamma)_p.
\end{align*}
First observe that the $p$-adic valuation of $d$ in $\Q_p$ is zero. This is because the minimal polynomial of $\gamma$, which is $16X^2-20X+5=0$, has two distinct non-zero roots mod $p$. This implies that $(\gamma,\gamma)_p=1$.

Next we observe that $-\gamma$ is not a square in $\Q_p$. If $-\gamma$ were a square in $\Q_p$, then $\Q_p$ would contain a primitive 5th root of unity. In particular, if $s\in \Q_p$ satisfies $s^2=-\gamma$, then one has that $\zeta=c+s$ is a root of the polynomial $X^4+X^3+X^2+X+1$, where $c=\frac{-1+\sqrt{5}}{4}$. One should expect this to be the case because $c=\frac{-1+\sqrt{5}}{4}=\cos \frac{2\pi}{5}$ and $\gamma=1-c^2$ so the quantity $s=\sqrt{c^2-1}$ behaves like $i\sin\frac{2\pi}{5}$. However, $\Q_p$ contains a primitive $5$th root of unity if and only if $p\equiv 1 \bmod 5$.

Thus for $p\equiv 4\bmod 5$, we have that
\[
\varepsilon_p(f)= (a^2 + 3ab + b^2, -\gamma)_p,
\]
where $-\gamma$ is a non-square in $\Q_p$ with zero $p$-adic valuation. It follows that $\varepsilon_p(f)$ is purely governed by the parity of the $p$-adic valuation of $a^2 + 3ab + b^2$, as required.
\end{proof}

Consider the rational quadratic form $\delta(x,y)=x^2+3xy+y^2$. We have $d(\delta)=-5$ and $\varepsilon_p(\delta)=1$ for all $p$. Let $q_1,\dots,q_\ell$ be a collection of distinct primes such that $q_i\equiv 4\bmod 5$. Since $(q_i,5)_p=1$ for all $p$, Theorem~\ref{thm:rank2_realization} implies that there are $a,b$ such that
\[
a^2+3ab+b^2=q_1\dots q_\ell.
\]
Moreover, we may assume that $a+b>0$. For such $a$ and $b$, we have that $f$ is positive definite and that for any prime $p\equiv 4\bmod 5$ we have that
\[
\varepsilon_p(f)=\begin{cases}
    -1 & \text{if $p\in \{q_1, \dots, q_\ell\}$}\\
    1 &\text{otherwise.}
\end{cases}
\]
This completes the verification of Lemma~\ref{lem:dim4analysis}\eqref{it:4-3'}.

\section{Further data}
Finally, we present a more complete breakdown of the data on the occurrence of flat manifolds of dimensions $n=3,4,5,6$ in terms of the holonomy group.
\begin{longtable}{|c | c | c | c | c | c|}
\hline
$|H|$ & $H$ & $\mathcal{C}$ & Type & \# Manifolds & \# Orientable \\
\hline
1 & \(1\) & 1,1,1  & \eqref{it:dim3none} & 1 & 1 \\
\hline
\hline
2 & \(C_2\) & 1,1,1 & \eqref{it:dim3none} & 3 & 1 \\
\hline
\hline
3 & \(C_3\) & 2-2,1 & \eqref{it:dim3mod3} & 1 & 1 \\
\hline
\hline
4 & \(C_2^2\) & 1,1,1  & \eqref{it:dim3none} & 3 & 1 \\
\hline
 & \(C_4\) & 2-1,1 & \eqref{it:dim3mod4} & 1 & 1 \\
\hline
\hline
6 & \(C_6\) & 2-2,1 & \eqref{it:dim3mod3} & 1 & 1 \\
\hline
\caption{The data for flat 3-manifolds broken down by holonomy group.}
\label{tab:dim3}
\end{longtable}

\begin{longtable}{|c | c | c | c | c | c|}
\hline
$|H|$ & $H$ & $\mathcal{C}$ & Type & \# Manifolds & \# Orientable \\
\hline
1 & \(1\) & 1,1,1,1 & \eqref{it:dim4none} & 1 & 1 \\
\hline
\hline
2 & \(C_2\) & 1,1,1,1 & \eqref{it:dim4none} & 5 & 2 \\
\hline
\hline
3 & \(C_3\) & 2-2,1,1 & \eqref{it:dim4mod3} & 2 & 2 \\
\hline
\hline
4 & \(C_2^2\) & 1,1,1,1 & \eqref{it:dim4none} & 26 & 9 \\
\hline
 & \(C_4\) & 2-1,1,1 & \eqref{it:dim4mod4} & 6 & 2 \\
\hline
\hline
6 & \(C_6\) & 2-2,1,1 & \eqref{it:dim4mod3} & 4 & 1 \\
\hline
 & \(S_3\) & 2-2,1,1 & \eqref{it:dim4mod3} & 3 & 3 \\
\hline
\hline
8 & \(C_2^3\) & 1,1,1,1 & \eqref{it:dim4none} & 12 & 0 \\
\hline
 & \(C_4 \times C_2\) & 2-1,1,1 & \eqref{it:dim4mod4} & 3 & 0 \\
\hline
 & \(D_8\) & 2-1,1,1 & \eqref{it:dim4mod4} & 7 & 4 \\
\hline
\hline
12 & \(A_4\) & 3,1& \eqref{it:dim4modother} & 2 & 2 \\
\hline
 & \(C_6 \times C_2\) & 2-2,1,1 & \eqref{it:dim4mod3} & 1 & 0 \\
\hline
 & \(D_{12}\) & 2-2,1,1 & \eqref{it:dim4mod3} & 1 & 1 \\
\hline
\hline
24 & \(C_2 \times A_4\) & 3,1& \eqref{it:dim4modother} & 1 & 0 \\
\hline
\caption{The data for flat 4-manifolds broken down by holonomy group.}
\label{tab:dim4}
\end{longtable}

\begin{longtable}{|c | c | c | c | c | c|}
\hline
$|H|$ & $H$ & $\mathcal{C}$ & Type & \# Manifolds & \# Orientable \\
\hline
\hline
1 & \(1\) & 1,1,1,1,1  & \eqref{it:dim5none} & 1 & 1 \\
\hline
\hline
2 & \(C_2\) &1,1,1,1,1  & \eqref{it:dim5none} & 8 & 4 \\
\hline
\hline
3 & \(C_3\) & 2-2,1,1,1& \eqref{it:dim5none} & 2 & 2 \\
\hline
 & \(C_3\) & 2-2,2-2,1& \eqref{it:dim5 2-2,2-2} & 1 & 1 \\
\hline
\hline
4 & \(C_2^2\) &1,1,1,1,1  & \eqref{it:dim5none} & 121 & 45 \\

\hline
 & \(C_4\) & 2-1,1,1,1& \eqref{it:dim5none} & 18 & 7 \\
\hline
 & \(C_4\) & 2-1,2-1,1& \eqref{it:dim5 2-1,2-1} & 1 & 1 \\
\hline
\hline
5 & \(C_5\) & 4-3',1 & \eqref{it:dim5 4-3p} & 1 & 1 \\
\hline
\hline
6 & \(C_6\) & 2-2,1,1,1& \eqref{it:dim5none} & 12 & 4 \\
\hline
 & \(C_6\) & 2-2,2-2,1& \eqref{it:dim5 2-2,2-2} & 3 & 3 \\
\hline
 & \(S_3\) & 2-2,1,1,1& \eqref{it:dim5none} & 10 & 6 \\
\hline
\hline
8 & \(C_2^3\) &1,1,1,1,1  & \eqref{it:dim5none} & 398 & 33 \\
\hline
 & \(C_4 \times C_2\) & 2-1,1,1,1& \eqref{it:dim5none} & 70 & 6 \\
\hline
 & \(C_4 \times C_2\) & 2-1,2-1,1& \eqref{it:dim5 2-1,2-1} & 3 & 3 \\
\hline
 & \(C_8\) & 4-1',1 & \eqref{it:dim5 4-1p} & 1 & 1 \\
\hline
 & \(D_8\) & 2-1,1,1,1& \eqref{it:dim5none} & 114 & 32 \\
\hline
 & \(D_8\) & 2-1,2-1,1& \eqref{it:dim5 2-1,2-1} & 1 & 0 \\
\hline
\hline
9 & \(C_3^2\) & 2-2,2-2,1& \eqref{it:dim5 2-2,2-2} & 3 & 3 \\
\hline
\hline
10 & \(C_{10}\) & 4-3',1 & \eqref{it:dim5 4-3p} & 1 & 1 \\
\hline
\hline
12 & \(A_4\) & 3,1,1& \eqref{it:dim5none} & 2 & 2 \\
\hline
 & \(C_6 \times C_2\) & 2-2,1,1,1& \eqref{it:dim5none} & 17 & 1 \\
\hline
 & \(C_{12}\) & 2-1,2-2,1& \eqref{it:dim5 2-1,2-2} & 2 & 2 \\
\hline
 & \(C_{12}\) & 4-2',1 & \eqref{it:dim5 4-2p} & 1 & 1 \\
\hline
 & \(D_{12}\) & 2-2,1,1,1& \eqref{it:dim5none} & 33 & 4 \\
\hline
\hline
16 & \((C_4 \times C_2) \rtimes C_2\) & 2-1,2-1,1& \eqref{it:dim5 2-1,2-1} & 13 & 2 \\
\hline
 & \(C_2 \times D_8\) & 2-1,1,1,1& \eqref{it:dim5none} & 59 & 0 \\
\hline
 & \(C_2^4\) &1,1,1,1,1  & \eqref{it:dim5none} & 123 & 2 \\
\hline
 & \(C_4 \times C_2^2\) & 2-1,1,1,1& \eqref{it:dim5none} & 16 & 0 \\
\hline
\hline
18 & \(C_3 \times S_3\) & 2-2,2-2,1& \eqref{it:dim5 2-2,2-2} & 4 & 0 \\
\hline
 & \(C_6 \times C_3\) & 2-2,2-2,1& \eqref{it:dim5 2-2,2-2} & 1 & 1 \\
\hline
\hline
24 & \(C_2 \times A_4\) & 3,1,1& \eqref{it:dim5none} & 7 & 2 \\
\hline
 & \(C_2^2 \times S_3\) & 2-2,1,1,1& \eqref{it:dim5none} & 4 & 0 \\
\hline
 & \(C_6 \times C_2^2\) & 2-2,1,1,1& \eqref{it:dim5none} & 2 & 0 \\
\hline
 & \(C_{12} \times C_2\) & 2-1,2-2,1& \eqref{it:dim5 2-1,2-2} & 1 & 1 \\
\hline
 & \(S_4\) & 3,1,1& \eqref{it:dim5none} & 4 & 2 \\
\hline
\hline
48 & \(C_2^2 \times A_4\) & 3,1,1& \eqref{it:dim5none} & 2 & 0 \\
\hline
\caption{The data for flat 5-manifolds broken down by holonomy group.}
\label{tab:dim5}
\end{longtable}

\begin{longtable}{|c | c | c | c | c | c|}
\hline
$|H|$ & $H$ & $\mathcal{C}$ & Type & \# Manifolds & \# Orientable \\
\hline
\hline
1 & \(1\) & 1,1,1,1,1,1 & \eqref{it:dim6none} & 1 & 1 \\
\hline
\hline
2 & \(C_2\) & 1,1,1,1,1,1 & \eqref{it:dim6none} & 11 & 5 \\
\hline
\hline
3 & \(C_3\) & 2-2,1,1,1,1 & \eqref{it:dim6none} & 2 & 2 \\
\hline
 & \(C_3\) & 2-2,2-2,1,1 & \eqref{it:dim6 2-2,2-2} & 2 & 2 \\
\hline
\hline
4 & \(C_2^2\) & 1,1,1,1,1,1 & \eqref{it:dim6none} & 429 & 124 \\
\hline
 & \(C_4\) & 2-1,1,1,1,1 & \eqref{it:dim6none} & 37 & 20 \\
\hline
 & \(C_4\) & 2-1,2-1,1,1 & \eqref{it:dim6 2-1,2-1} & 6 & 2 \\
\hline
\hline
5 & \(C_5\) & 4-3',1,1 & \eqref{it:dim6 4-3p} & 2 & 2 \\
\hline
\hline
6 & \(C_6\) & 2-2,1,1,1,1 & \eqref{it:dim6none} & 20 & 9 \\
\hline
 & \(C_6\) & 2-2,2-2,1,1 & \eqref{it:dim6 2-2,2-2} & 12 & 5 \\
\hline
 & \(S_3\) & 2-2,1,1,1,1 & \eqref{it:dim6none} & 18 & 10 \\
\hline
 & \(S_3\) & 2-2,2-2,1,1 & \eqref{it:dim6 2-2,2-2} & 5 & 0 \\
\hline
\hline
8 & \(C_2^3\) & 1,1,1,1,1,1 & \eqref{it:dim6none} & 7050 & 836 \\
\hline
 & \(C_4 \times C_2\) & 2-1,1,1,1,1 & \eqref{it:dim6none} & 736 & 228 \\
\hline
 & \(C_4 \times C_2\) & 2-1,2-1,1,1 & \eqref{it:dim6 2-1,2-1} & 76 & 19 \\
\hline
 & \(C_8\) & 4-1',1,1 & \eqref{it:dim6 4-1p} & 6 & 2 \\
\hline
 & \(D_8\) & 2-1,1,1,1,1 & \eqref{it:dim6none} & 1059 & 219 \\
\hline
 & \(D_8\) & 2-1,2-1,1,1 & \eqref{it:dim6 2-1,2-1} & 34 & 1 \\
\hline
 & \(Q_8\) & 4-1,1,1 & \eqref{it:dim6 4-1} & 4 & 0 \\
\hline
\hline
9 & \(C_3^2\) & 2-2,2-2,1,1 & \eqref{it:dim6 2-2,2-2} & 13 & 13 \\
\hline
\hline
10 & \(C_{10}\) & 4-3',1,1 & \eqref{it:dim6 4-3p} & 4 & 1 \\
\hline
 & \(D_{10}\) & 4-3',1,1 & \eqref{it:dim6 4-3p} & 3 & 0 \\
\hline
\hline
12 & \(A_4\) & 3,1,1,1 & \eqref{it:dim6none} & 2 & 2 \\
\hline
 & \(A_4\) & 3,2-2,1 & \eqref{it:dim6 2-2,2-2} & 14 & 14 \\
\hline
 & \(C_3 \rtimes C_4\) & 2-1,2-2,1,1& \eqref{it:dim6 2-1,2-2} & 12 & 12 \\
\hline
 & \(C_3 \rtimes C_4\) & 4-2,1,1 & \eqref{it:dim6 4-2} & 2 & 0 \\
\hline
 & \(C_6 \times C_2\) & 2-2,1,1,1,1 & \eqref{it:dim6none} & 153 & 48 \\
\hline
 & \(C_6 \times C_2\) & 2-2,2-2,1,1 & \eqref{it:dim6 2-2,2-2} & 11 & 2 \\
\hline
 & \(C_{12}\) & 2-1,2-2,1,1& \eqref{it:dim6 2-1,2-2} & 18 & 6 \\
\hline
 & \(C_{12}\) & 4-2',1,1 & \eqref{it:dim6 4-2p} & 3 & 1 \\
\hline
 & \(D_{12}\) & 2-2,1,1,1,1 & \eqref{it:dim6none} & 298 & 41 \\
\hline
 & \(D_{12}\) & 2-2,2-2,1,1 & \eqref{it:dim6 2-2,2-2} & 8 & 0 \\
\hline
\hline
16 & \((C_4 \times C_2) \rtimes C_2\) & 2-1,2-1,1,1 & \eqref{it:dim6 2-1,2-1} & 539 & 153 \\
\hline
 & \((C_4 \times C_2) \rtimes C_2\) & 2-1,2-1,2-1 & \eqref{it:dim6 2-1,2-1,2-1} & 4 & 0 \\
\hline
 & \((C_4 \times C_2) \rtimes C_2\) & 4-1,1,1 & \eqref{it:dim6 4-1} & 4 & 0 \\
\hline
 & \(C_2 \times D_8\) & 2-1,1,1,1,1 & \eqref{it:dim6none} & 4228 & 345 \\
\hline
 & \(C_2 \times D_8\) & 2-1,2-1,1,1 & \eqref{it:dim6 2-1,2-1} & 69 & 2 \\
\hline
 & \(C_2^4\) & 1,1,1,1,1,1 & \eqref{it:dim6none} & 17488 & 780 \\
\hline
 & \(C_4 \rtimes C_4\) & 2-1,2-1,1,1 & \eqref{it:dim6 2-1,2-1} & 34 & 25 \\
\hline
 & \(C_4 \times C_2^2\) & 2-1,1,1,1,1 & \eqref{it:dim6none} & 1285 & 108 \\
\hline
 & \(C_4 \times C_2^2\) & 2-1,2-1,1,1 & \eqref{it:dim6 2-1,2-1} & 24 & 0 \\
\hline
 & \(C_4^2\) & 2-1,2-1,1,1 & \eqref{it:dim6 2-1,2-1} & 14 & 8 \\
\hline
 & \(C_8 \rtimes C_2\) & 4-1,1,1 & \eqref{it:dim6 4-1} & 7 & 0 \\
\hline
 & \(C_8 \times C_2\) & 4-1',1,1 & \eqref{it:dim6 4-1p} & 3 & 0 \\
\hline
 & \(D_{16}\) & 4-1',1,1 & \eqref{it:dim6 4-1p} & 5 & 0 \\
\hline
 & \(QD_{16}\) & 4-1,1,1 & \eqref{it:dim6 4-1} & 3 & 0 \\
\hline
\hline
18 & \(C_3 \times S_3\) & 2-2,2-2,1,1 & \eqref{it:dim6 2-2,2-2} & 36 & 31 \\
\hline
 & \(C_3 \times S_3\) & 4-2,1,1 & \eqref{it:dim6 4-2} & 5 & 0 \\
\hline
 & \(C_3^2 \rtimes C_2\) & 2-2,2-2,1,1 & \eqref{it:dim6 2-2,2-2} & 6 & 0 \\
\hline
 & \(C_6 \times C_3\) & 2-2,2-2,1,1 & \eqref{it:dim6 2-2,2-2} & 11 & 6 \\
\hline
\hline
20 & \(C_{10} \times C_2\) & 4-3',1,1 & \eqref{it:dim6 4-3p} & 1 & 0 \\
\hline
 & \(D_{20}\) & 4-3',1,1 & \eqref{it:dim6 4-3p} & 1 & 0 \\
\hline
\hline
24 & \((C_6 \times C_2) \rtimes C_2\) & 2-1,2-2,1,1& \eqref{it:dim6 2-1,2-2} & 15 & 5 \\
\hline
 & \(C_2 \times (C_3 \rtimes C_4)\) & 2-1,2-2,1,1& \eqref{it:dim6 2-1,2-2} & 8 & 8 \\
\hline
 & \(C_2 \times A_4\) & 3,1,1,1 & \eqref{it:dim6none} & 21 & 9 \\
\hline
 & \(C_2 \times A_4\) & 3,2-2,1 & \eqref{it:dim6 2-2,2-2} & 34 & 16 \\
\hline
 & \(C_2^2 \times S_3\) & 2-2,1,1,1,1 & \eqref{it:dim6none} & 502 & 30 \\
\hline
 & \(C_3 \rtimes C_8\) & 4-1,1,1 & \eqref{it:dim6 4-1} & 2 & 0 \\
\hline
 & \(C_3 \times D_8\) & 2-1,2-2,1,1& \eqref{it:dim6 2-1,2-2} & 30 & 20 \\
\hline
 & \(C_3 \times D_8\) & 4-2,1,1 & \eqref{it:dim6 4-2} & 1 & 0 \\
\hline
 & \(C_3 \times Q_8\) & 4-1,1,1 & \eqref{it:dim6 4-1} & 1 & 0 \\
\hline
 & \(C_4 \times S_3\) & 2-1,2-2,1,1& \eqref{it:dim6 2-1,2-2} & 15 & 15 \\
\hline
 & \(C_6 \times C_2^2\) & 2-2,1,1,1,1 & \eqref{it:dim6none} & 164 & 19 \\
\hline
 & \(C_{12} \times C_2\) & 2-1,2-2,1,1& \eqref{it:dim6 2-1,2-2} & 27 & 4 \\
\hline
 & \(D_{24}\) & 2-1,2-2,1,1& \eqref{it:dim6 2-1,2-2} & 25 & 0 \\
\hline
 & \(D_{24}\) & 4-2',1,1 & \eqref{it:dim6 4-2p} & 1 & 0 \\
\hline
 & \(S_4\) & 3,1,1,1 & \eqref{it:dim6none} & 13 & 5 \\
\hline
 & \(S_4\) & 3,2-2,1 & \eqref{it:dim6 2-2,2-2} & 6 & 0 \\
\hline
\hline
32 & \(((C_4 \times C_2) \rtimes C_2) \rtimes C_2\) & 4-1,2-1 & \eqref{it:dim6 2-1,2-1,2-1} & 2 & 0 \\
\hline
 & \((C_4 \times C_2^2) \rtimes C_2\) & 2-1,2-1,1,1 & \eqref{it:dim6 2-1,2-1} & 13 & 7 \\
\hline
 & \((C_8 \rtimes C_2) \rtimes C_2\) & 4-1,1,1 & \eqref{it:dim6 4-1} & 1 & 0 \\
\hline
 & \(C_2 \times ((C_4 \times C_2) \rtimes C_2)\) & 2-1,2-1,1,1 & \eqref{it:dim6 2-1,2-1} & 62 & 0 \\
\hline
 & \(C_2^2 \times D_8\) & 2-1,1,1,1,1 & \eqref{it:dim6none} & 871 & 16 \\
\hline
 & \(C_2^4 \rtimes C_2\) & 2-1,2-1,1,1 & \eqref{it:dim6 2-1,2-1} & 172 & 3 \\
\hline
 & \(C_2^4 \rtimes C_2\) & 2-1,2-1,2-1 & \eqref{it:dim6 2-1,2-1,2-1} & 1 & 0 \\
\hline
 & \(C_2^5\) & 1,1,1,1,1,1 & \eqref{it:dim6none} & 2536 & 0 \\
\hline
 & \(C_4 \times C_2^3\) & 2-1,1,1,1,1 & \eqref{it:dim6none} & 184 & 0 \\
\hline
 & \(C_4 \times D_8\) & 2-1,2-1,1,1 & \eqref{it:dim6 2-1,2-1} & 38 & 32 \\
\hline
 & \(C_4^2 \rtimes C_2\) & 4-1,1,1 & \eqref{it:dim6 4-1} & 4 & 0 \\
\hline
\hline
36 & \(C_2 \times (C_3^2 \rtimes C_2)\) & 2-2,2-2,1,1 & \eqref{it:dim6 2-2,2-2} & 1 & 0 \\
\hline
 & \(C_3 \times (C_3 \rtimes C_4)\) & 4-2,1,1 & \eqref{it:dim6 4-2} & 2 & 0 \\
\hline
 & \(C_3 \times A_4\) & 3,2-2,1 & \eqref{it:dim6 2-2,2-2} & 2 & 2 \\
\hline
 & \(C_6 \times S_3\) & 2-2,2-2,1,1 & \eqref{it:dim6 2-2,2-2} & 21 & 13 \\
\hline
 & \(C_6 \times S_3\) & 4-2,1,1 & \eqref{it:dim6 4-2} & 2 & 0 \\
\hline
 & \(C_6^2\) & 2-2,2-2,1,1 & \eqref{it:dim6 2-2,2-2} & 2 & 1 \\
\hline
 & \(S_3^2\) & 2-2,2-2,1,1 & \eqref{it:dim6 2-2,2-2} & 7 & 0 \\
\hline
\hline
48 & \(C_2 \times C_4 \times S_3\) & 2-1,2-2,1,1& \eqref{it:dim6 2-1,2-2} & 3 & 3 \\
\hline
 & \(C_2 \times D_{24}\) & 2-1,2-2,1,1& \eqref{it:dim6 2-1,2-2} & 2 & 0 \\
\hline
 & \(C_2 \times S_4\) & 3,1,1,1 & \eqref{it:dim6none} & 28 & 6 \\
\hline
 & \(C_2 \times S_4\) & 3,2-2,1 & \eqref{it:dim6 2-2,2-2} & 3 & 0 \\
\hline
 & \(C_2^2 \times A_4\) & 3,1,1,1 & \eqref{it:dim6none} & 40 & 5 \\
\hline
 & \(C_2^2 \times A_4\) & 3,2-2,1 & \eqref{it:dim6 2-2,2-2} & 2 & 0 \\
\hline
 & \(C_2^3 \times S_3\) & 2-2,1,1,1,1 & \eqref{it:dim6none} & 24 & 0 \\
\hline
 & \(C_4 \times A_4\) & 3,2-1,1 & \eqref{it:dim6 2-1,2-1} & 5 & 3 \\
\hline
 & \(C_6 \times C_2^3\) & 2-2,1,1,1,1 & \eqref{it:dim6none} & 8 & 0 \\
\hline
 & \(C_6 \times D_8\) & 2-1,2-2,1,1& \eqref{it:dim6 2-1,2-2} & 2 & 2 \\
\hline
 & \(C_{12} \times C_2^2\) & 2-1,2-2,1,1& \eqref{it:dim6 2-1,2-2} & 3 & 0 \\
\hline
\hline
64 & \(((C_8 \rtimes C_2) \rtimes C_2) \rtimes C_2\) & 4-1,2-1 & \eqref{it:dim6 2-1,2-1,2-1} & 1 & 0 \\
\hline
\hline
72 & \(C_3 \times S_4\) & 3,2-2,1 & \eqref{it:dim6 2-2,2-2} & 2 & 1 \\
\hline
 & \(C_6 \times A_4\) & 3,2-2,1 & \eqref{it:dim6 2-2,2-2} & 1 & 0 \\
\hline
\hline
80 & \(C_2^4 \rtimes C_5\) & 5-1,1 & \eqref{it:dim6 4-1} & 4 & 4 \\
\hline
\hline
96 & \(C_2 \times C_4 \times A_4\) & 3,2-1,1 & \eqref{it:dim6 2-1,2-1} & 2 & 0 \\
\hline
 & \(C_2^3 \times A_4\) & 3,1,1,1 & \eqref{it:dim6none} & 8 & 0 \\
\hline
\hline
160 & \(C_2 \times (C_2^4 \rtimes C_5)\) & 5-1,1 & \eqref{it:dim6 4-1} & 4 & 0 \\
\hline

\caption{The data for flat 6-manifolds broken down by holonomy group.}
\label{tab:dim6}
\end{longtable}

\end{document}